\documentclass[11pt,twoside]{amsart}

\title{Lefschetz operators on convex valuations}

\author{Leo Brauner}
\address{Institute of Discrete Mathematics and Geometry \newline%
	 \indent Vienna University of Technology \newline%
	 \indent 1040 Wien, Austria}
\email{leo.brauner@tuwien.ac.at}

\author{Georg C.\ Hofstätter}
\address{Institute for Mathematics \newline%
	\indent Friedrich-Schiller-University Jena \newline%
	\indent 07743 Jena, Germany}
\email{georg.hofstaetter@uni-jena.de}

\author{Oscar Ortega-Moreno}
\address{Institute of Discrete Mathematics and Geometry \newline%
	\indent Vienna University of Technology \newline%
	\indent 1040 Wien, Austria}
\email{oscar.moreno@tuwien.ac.at}

\usepackage[english]{babel}
\usepackage[T1]{fontenc}
\usepackage{lmodern}
\usepackage{amsmath,amsthm,amssymb,amsfonts}
\usepackage[sort&compress,numbers]{natbib}
\usepackage{hyperref}
\usepackage[capitalise]{cleveref}
\usepackage[left=4cm,right=4cm,top=3cm,bottom=3cm,includeheadfoot]{geometry}
\usepackage{bbm}
\usepackage{xcolor}
\usepackage{enumerate,enumitem}
\usepackage{fancyhdr}
\usepackage{tikz}
\usepackage{float}

\hypersetup{colorlinks=true,
	linkcolor=teal,
	citecolor=orange}

\theoremstyle{plain}

\newtheorem{thmintro}{Theorem}
\newtheorem*{thmintro*}{Theorem}

\newtheorem*{corintro*}{Corollary}

\theoremstyle{plain}
\newtheorem{lem}{Lemma}[section]
\newtheorem{prop}[lem]{Proposition}
\newtheorem{thm}[lem]{Theorem}
\newtheorem*{thm*}{Theorem}
\newtheorem{cor}[lem]{Corollary}

\theoremstyle{definition}
\newtheorem{defi}[lem]{Definition}

\theoremstyle{remark}
\newtheorem{exl}[lem]{Example}
\newtheorem{rem}[lem]{Remark}

\numberwithin{equation}{section}

\crefname{equation}{}{}
\crefname{lem}{Lemma}{Lemmas}
\crefname{thm}{Theorem}{Theorems}
\crefname{rem}{Remark}{Remarks}
\crefname{enumi}{part}{parts}

\newcommand{\N}{\mathbb{N}}			
\newcommand{\R}{\mathbb{R}}			
\renewcommand{\S}{\mathbb{S}}		
\renewcommand{\H}{\mathbb{H}}		
\newcommand{\Gr}{\mathrm{Gr}}		
\newcommand{\AGr}{\overline{\mathrm{Gr}}}		
\newcommand{\Fl}{\mathrm{Fl}}		
\newcommand{\D}{\mathcal{D}}		
\newcommand{\SO}{\mathrm{SO}}		

\newcommand{\CC}{\mathcal{C}}		

\newcommand{\K}{\mathcal{K}}		
\newcommand{\Val}{\mathbf{Val}}		
\newcommand{\MVal}{\mathbf{MVal}}	

\newcommand{\MSO}{\mathrm{M}}		

\newcommand{\abs}[1]{\lvert #1 \rvert}						
\newcommand{\norm}[1]{\lVert #1 \rVert}						
\newcommand{\restr}[2]{\left. #1 \right|_{#2}}				
\newcommand{\flag}[2]{\genfrac{[}{]}{0pt}{}{#1}{#2}}		
\newcommand{\pr}{\mathrm{pr}}								
\newcommand{\Id}{\mathrm{Id}}								

\newcommand{\RT}{\tilde{\mathcal{R}}}	
\newcommand{\LL}{\mathfrak{L}}			

\newcommand{\Kl}{\mathrm{Kl}}		
\newcommand{\KS}{\mathrm{KS}}		

\begin{document}
	
	\begin{abstract}
		We investigate the action of Alesker's Lefschetz operators on translation invariant valuations on convex bodies. For scalar valued valuations, we describe this action on the level of Klain--Schneider functions by a Radon type transform, generalizing a result by Schuster and Wannerer. In the case of rotationally equivariant Minkowski valuations, the Lefschetz operators act on the generating function as a convolution transform. We show that the convolution kernel satisfies a Legendre type differential equation, and thus, is a strictly positive function that is smooth up to one point.
	\end{abstract}
	
	\maketitle
	\thispagestyle{empty}

	\section{Introduction}
	\label{sec:intro}

	\subsection*{Scalar valued valuations}

	A \emph{valuation} on the space $\K(\R^n)$ of convex bodies (convex, compact subsets) in~$\R^n$ is a functional $\varphi:\K(\R^n) \to\R$ such that
	\begin{equation*}
		\varphi(K \cup L) + \varphi(K \cap L) = \varphi(K) + \varphi(L)
	\end{equation*}
	whenever $K, L, K\cup L \in \K(\R^n)$.
	Valuations have a long history in convex and integral geometry (see, e.g., \cite{Hadwiger1957,Ludwig2010a,Alesker1999,Bernig2011,Klain1997,Alesker2001,Alesker2018,Haberl2014,Alesker2003,Faifman2023}). We denote by $\Val$ the space of continuous, translation invariant valuations. By a classical result of McMullen~\cite{McMullen1980}, the space $\Val$ is the direct sum of the subspaces $\Val_i$ of valuations that are homogeneous of degree $i\in\{0,\ldots,n\}$
	(that is, $\varphi(\lambda K)=\lambda^i \varphi(K)$ for all $K\in\K(\R^n)$ and $\lambda\geq 0$).
	
	Motivated by the Hard Lefschetz theorem from Kähler geometry, Alesker~\cite{Alesker2003} introduced \emph{Lefschetz operators} on valuations. For $\varphi\in\Val$ and $K\in\K(\R^n)$,
	\begin{equation*}
		(\Lambda\varphi)(K)
		= \left. \frac{d}{dt} \right|_{t=0^+} \!\!\varphi(K+tB^n)
		\qquad\text{and}\qquad
		(\LL\varphi)(K)
		= \int_{\AGr_{n,n-1}} \!\! \varphi(K\cap H) ~dH,
	\end{equation*}
	where $B^n$ is the unit ball in $\R^n$ and $\AGr_{n,j}$ is the Grassmann manifold of affine \linebreak $j$-dimensional subspaces of $\R^n$, endowed with a (suitably normalized) rigid motion invariant measure.	
	The Lefschetz operators are a powerful tool in valuation theory, since they allow to transfer results between valuations of different degrees (see, e.g., \cite{Alesker2003,Parapatits2012,Berg2018,Bernig2007,Schuster2006,Bernig2024,Kotrbaty2023,Kotrbaty2022,Kotrbaty2021,Alesker2022}).
	The derivation operator $\Lambda$ decreases the degree of a valuation by one; the integral operator $\LL$ increases it by one.
	
	In this article, we investigate the action of the Lefschetz operators on two well-known representations of valuations: Klain--Schneider functions for scalar valued valuations and generating functions for Minkowski valuations.	
	Denoting by $\Gr_{n,i}$ the Grassmann manifold of $i$-dimensional subspaces of $\R^n$, Klain~\cite{Klain1995} showed that if $\varphi\in\Val_i$ is \emph{even} and $E\in\Gr_{n,i}$, then $\restr{\varphi}{E}=c_E \mathrm{vol}_E$ for some $c_E\in\R$, where $\restr{\varphi}{E}$ is the restriction of $\varphi$ to convex bodies in $E$.
	Its \emph{Klain function} $\Kl_{\varphi}\in C(\Gr_{n,i})$, $\Kl_{\varphi}: E\mapsto c_E$, uniquely determines $\varphi\in\Val_i$, as was proved by Klain~\cite{Klain2000}.
	
%
	
	Schuster and Wannerer~\cite{Schuster2015} showed that for \emph{smooth valuations} (see \cref{sec:Val_sph}), the Lefschetz operators act on the Klain function by a Radon transform between the different Grassmannians. In the following, $\kappa_k$ denotes the volume of $B^k$ and $\Gr_{n,i}^E\subseteq\Gr_{n,i}$ denotes the submanifold of spaces that are contained in $E$ or that contain $E$, depending on the dimension of~$E$. Integration is with respect to the unique probability measure invariant under rotations fixing $E$.
	
	\begin{thm*}[{\cite{Schuster2015}}] \label{LL_Lambda_Kl}
		Let $1\leq i\leq  n-1$ and $\varphi\in\Val_i$ be smooth and even.
		\begin{enumerate}[label=\upshape(\roman*)]
			\item $\displaystyle \Kl_{\Lambda\varphi}(E)
			= \frac{(n-i+1)\kappa_{n-i+1}}{\kappa_{n-i}} \! \int_{\Gr_{n,i}^E} \!\! \Kl_{\varphi}(F) dF$ for $E\in\Gr_{n,i-1}$, if $i>1$.
			\item $\displaystyle \Kl_{\LL\varphi}(E)
			= \frac{(i+1)\kappa_{n-1}\kappa_{i+1}}{n\kappa_n\kappa_i} \! \int_{\Gr_{n,i}^E} \!\! \Kl_{\varphi}(F) dF$ for $E\in\Gr_{n,i+1}$, if $i<n-1$.
		\end{enumerate}
	\end{thm*}
	
	In the odd case, Schneider~\cite{Schneider1996} introduced the natural counterpart to the Klain function, called the \emph{Schneider function}, and showed that it determines the respective valuation uniquely.
	We combine these representations for the even and odd part of a valuation $\varphi\in\Val_i$ into one common continuous function on the \emph{flag manifold}
	$\Fl_{n,i+1}=\{(E,u): E\in\Gr_{n,i+1}, u\in\S^i(E)\}$, where $\S^{i}(E)$ denotes the unit sphere in $E$.
	This function, called its \emph{Klain--Schneider function} and denoted by $\KS_{\varphi}$, has the property that for every $E\in\Gr_{n,i+1}$,
	\begin{equation*}
		\varphi(K)
		= \int_{\S^i(E)} \KS_{\varphi}(E,u) ~S_i^E(K,du),
		\qquad K\in\K(E),
	\end{equation*}
	where $S_i^{E}(K,{}\cdot{})$ is the surface area measure of $K$ relative to $E$ (cf. \cite[Section~4.2]{Schneider2014}). This determines $\KS_{\varphi}(E,{}\cdot{})$ only up to the addition of linear functions, so for now, we define it as the corresponding equivalence class.
	By the results of Klain and Schneider, $\varphi$ is uniquely determined by $\KS_{\varphi}$.
	
	With our first main result, we describe the action of the Lefschetz operators on the Klain--Schneider function, generalizing the theorem above.
	Here, ${}\cdot{}|E$ denotes the orthogonal projection onto a linear subspace $E\subseteq\R^n$ and $\pr_{\! E} u = \norm{u|E}^{-1} u|E$ for $u\in\S^{n-1}\setminus E^{\perp}$.
	
	\begin{thmintro}\label{LL_Lambda_KS:intro}
		Let $1\leq i\leq n-1$ and $\varphi \in\Val_i$.
		\begin{enumerate}[label=\upshape(\roman*)]
			\item \label{Lambda_KS:intro}
			If $i>1$, then for all $(E,u)\in\Fl_{n,i}$,
			\begin{equation*}
				\KS_{\Lambda\varphi}(E,u)
				= \frac{(n-i+1)\kappa_{n-i+1}}{\kappa_{n-i}} \int_{\Gr^{E\cap u^\perp}_{n,i}} \KS_{\varphi}(\mathrm{span}(F\cup u),\pr_{\! F^\perp}u )dF.
			\end{equation*}
			\item \label{LL_KS:intro}
			If $i<n-1$, then for all $(E,u)\in\Fl_{n,i+2}$,
			\begin{equation*}
				\KS_{\LL\varphi}(E,u)
				= \frac{(i+2)\kappa_{n-1}\kappa_{i+2}}{n\kappa_n\kappa_{i+1}} \int_{\Gr_{n,i+1}^E} \!\!\! \KS_{\varphi}(F,\pr_{\! F} u )\norm{u | F} dF.
			\end{equation*}
		\end{enumerate}
	\end{thmintro}
	
	In fact, our results extend to more general Lefschetz operators: we can describe the Alesker product and the Bernig--Fu convolution with an even valuation of degree and codegree one, respectively (see \cref{rem:generalized_Lambda,rem:generalized_LL}).
	
	We also want to emphasize that our approach is very different from that in~\cite{Schuster2015}.	
	Our proof of~\ref{LL_KS:intro} only uses simple geometric properties of polytopes. For~\ref{Lambda_KS:intro}, we reduce the general case to valuations of the form $\varphi(K)=V(K^{[i]},\CC)$ defined in terms of a mixed volume with a family $\CC=(C_1,\ldots,C_{n-i})$ of fixed reference bodies.
	Indeed, as a consequence of his irreducibility theorem~\cite{Alesker2001}, Alesker proved that such valuations are dense in $\Val_i$, confirming a conjecture by McMullen~\cite{McMullen1980} (this was recently refined by Knoerr~\cite{Knoerr2023} for smooth valuations).
	
	The computation of the Klain--Schneider function of a mixed volume boils down to a relation between mixed area measures and surface are measures relative to a subspace. In order to establish the required relation, we generalize the \emph{spherical projections and liftings} that were introduced by Goodey, Kiderlen, and Weil~\cite{Goodey2011}.
	
	\begin{thmintro}\label{mixed_area_meas_lift:intro}
		Let $0\leq i\leq n-1$, $E\in\Gr_{n,i+1}$, and $\CC=(C_1,\ldots,C_{n-i-1})$ be a family of convex bodies with $C^2$~support functions. Then for all $K\in\K(E)$,
		\begin{equation}\label{eq:mixed_area_meas_lift:intro}
			S(K^{[i]},\CC,{}\cdot{})
			= \frac{1}{\binom{n-1}{i}} \pi_{E,\CC}^\ast S_i^E(K,{}\cdot{}).
		\end{equation}
	\end{thmintro}
	
	Here, $S(K^{[i]},\CC,{}\cdot{})$ denotes the mixed area measure and $\pi_{E,\CC}^\ast$ denotes the \linebreak \emph{$\CC$-mixed spherical lifting}, which we define in \cref{sec:proj_lift} as a linear operator mapping measures on $\S^i(E)$ to measures on $\S^{n-1}$. The special case where the reference bodies are balls was treated in~\cite{Goodey2011}.

	\subsection*{Minkowski valuations}

	We now turn to valuations that are convex body valued:
	A \emph{Minkowski valuation} is a map $\Phi:\K(\R^n)\to\K(\R^n)$ such that
	\begin{equation*}
		\Phi(K\cup L) + \Phi(K\cap L)
		= \Phi K + \Phi L
	\end{equation*}
	whenever $K,L,K\cup L\in\K(\R^n)$, where addition on $\K(\R^n)$ is the usual Minkowski addition.
	In the last two decades, starting with the seminal work of Ludwig~\cite{Ludwig2005,Ludwig2002}, considerable advances in the theory of Minkowski valuations have been made,
	including classification results and isoperimetric type inequalities (see, e.g., \cite{Berg2018,OrtegaMoreno2021,OrtegaMoreno2023,Ludwig2006,Parapatits2012,Hofstaetter2023,Brauner2023,Schuster2006,Haberl2012}).
	We denote by $\MVal$ the space of continuous, translation invariant Minkowski valuations that intertwine all rotations and by $\MVal_i$ the subspace of $i$-homogeneous Minkowski valuations. 
	
	Through the correspondence between a convex body $K\in\K(\R^n)$ and its support function $h_K \in C(\S^{n-1})$, $h_K(u)=\max_{x\in K}\langle u,x\rangle$ with the Euclidean inner product $\langle {}\cdot{},{}\cdot{}\rangle$, it is natural to define the Lefschetz operators on the space $\MVal$ by
	\begin{equation*}
		h_{(\Lambda\Phi)(K)}
		= \left. \frac{d}{dt} \right|_{t=0^+} \!\! h_{\Phi(K+tB^n)}
		\qquad\text{and}\qquad
		h_{(\LL\Phi)(K)}
		= \int_{\AGr_{n,n-1}} \!\! h_{\Phi(K\cap H)} ~dH,
	\end{equation*} 
	where these identities are to be understood pointwise. The integration operator~$\LL$ clearly preserves convexity, whereas the non-trivial fact that the derivation operator $\Lambda$ is well-defined on $\MVal$ is due to Parapatits and Schuster~\cite{Parapatits2012}.
	
	Every $\Phi\in\MVal$ is determined by its \emph{associated real valued valuation}, which is the $\SO(n-1)$-invariant valuation $\varphi\in\Val$ defined by $\varphi(K)=h_{\Phi K}(e_n)$, where $e_n$ is the north pole of $\S^{n-1}$. In this sense, Minkowski valuations present a special case of real valued valuations and in addition to the Klain--Schneider function, we have a more specific representation available involving the \emph{spherical convolution} (see \cref{sec:convol}) and the $i$-th order area measure $S_i(K,{}\cdot{})=S(K^{[i]},{B^n}^{[n-i-1]},{}\cdot{})$.
	
	\begin{thmintro*}[{\cite{Schuster2018,Dorrek2017}}] \label{gen_fct:intro}
		Let $1\leq i\leq n-1$ and $\Phi\in\MVal_i$. Then there exists a unique centered, $\SO(n-1)$-invariant function $f_\Phi\in L^1(\S^{n-1})$ such that
		\begin{equation} \label{eq:gen_fct:intro}
			h_{\Phi K}
			= S_i(K,{}\cdot{}) \ast f_\Phi,
			\qquad K\in\K(\R^n).
		\end{equation}
	\end{thmintro*}
	
	The function $f_\Phi$ is called the \emph{generating function} of $\Phi$, and by \emph{centered}, we mean that $f_{\Phi}$ is orthogonal to all linear functions.
	The representation above was first established by Schuster and Wannerer~\cite{Schuster2018} for a measure $f_{\Phi}$ and then improved by Dorrek~\cite{Dorrek2017} to $f_{\Phi}\in L^1(\S^{n-1})$. More recently, the first- and third-named author~\cite{Brauner2023} refined this further by showing that~$f_\Phi$ is in fact locally Lipschitz continuous outside the poles $\pm e_n$.
	
	The action of the Lefschetz derivation operator $\Lambda$ on the generating function is a simple multiplication by a constant, which is a direct consequence of the Steiner formula for area measures. For the integration operator $\LL$ however, things are far more intricate: For $\Phi\in\MVal_i$,
	\begin{equation}\label{eq:LL_Lambda_gen_fct:intro}
		f_{\Lambda\Phi}
		= i f_{\Phi}
		\qquad\text{and}\qquad
		f_{\LL\Phi}
		= f_{\Phi} \ast \rho_i
	\end{equation}
	with some unique $\SO(n-1)$-invariant, centered distribution $\rho_i$ on $\S^{n-1}$, as was shown in~\cite{Schuster2018}. Since this is a purely structural statement, describing the action of the Lefschetz integration operator $\LL$ on Minkowski valuations boils down to describing the distribution $\rho_i$. This is the content of our second main result.
	
	\begin{thmintro}\label{rho_description:intro}
		Let $1 \leq i < n-1$. Then $\rho_i$ is an $L^1(\S^{n-1})$ function that is smooth on $\S^{n-1}\setminus \{e_n\}$ and strictly positive up to the addition of a linear function.
		Moreover, the function $\bar\rho_i$, defined by $\rho_i(u) = \bar{\rho}_{i}(\langle e_n, u\rangle)$, satisfies for all $t\in (-1,1)$,
		\begin{equation}\label{eq:rho_ODE:intro}
			(1-t^2)\frac{d^2}{dt^2} \bar\rho_{i}(t) - n t \frac{d}{dt}\bar\rho_{i}(t) - i(n-i-1)\bar\rho_{i}(t)
			= 0.
		\end{equation}
	\end{thmintro}
	
	In particular, we obtain a representation of $\rho_i$ in terms of a hypergeometric function. The positivity of $\rho_i$ has several notable consequences. Whenever $\Phi$ is \linebreak generated by the support function of a body of revolution, then so is $\LL\Phi$. This answers a question of Schuster \cite{Schuster2023}. Moreover, the combined Lefschetz operators $\Lambda\circ\LL$ and $\LL\circ\Lambda$ act on Minkowski valuations as a composition with some \emph{Minkowski endomorphism}, that is, a Minkowski valuation in $\MVal_1$.
	
	\begin{corintro*}
		Let $1\leq i< n-1$. There exists $\Psi^{(i)}\in\MVal_1$ such that
		\begin{equation*}
			\Lambda( \LL\Phi_i)
			= \Psi^{(i)}\circ\Phi_i \qquad \text{and} \qquad \LL(\Lambda\Phi_{i+1}) = \Psi^{(i)} \circ \Phi_{i+1}
		\end{equation*}
		for every $\Phi_i\in\MVal_i$ and $\Phi_{i+1}\in\MVal_{i+1}$.
	\end{corintro*}
	
	Subsequently, the examples of Minkowski valuations that are currently known are preserved under the Lefschetz operators (see \cref{sec:MVali,sec:rho_description}). 
	Moreover, we want to point out that the notion of generating functions and the validity of \cref{eq:LL_Lambda_gen_fct:intro} extends to the much larger class of \emph{spherical valuations} (see \cref{sec:Val_sph}), and thus, so does the scope of \cref{rho_description:intro}.

	\subsection*{Organization of the article}

	In \cref{sec:proj_lift}, we introduce mixed spherical projections and liftings and prove \cref{mixed_area_meas_lift:intro}. We apply this in \cref{sec:LL_Lambda_KS} to the Klain--Schneider function of scalar valued valuations; in there, we show \cref{LL_Lambda_KS:intro}. Then, in \cref{sec:MSOs} we turn to Minkowski valuations, discussing the structure of the space $\MVal$ and making some key preparations. \cref{sec:LL_gen_fct} is devoted to the action of the Lefschetz operators on generating functions; in there, we prove \cref{rho_description:intro} and its consequences. Finally, in \cref{sec:KS_gen_fct}, we discuss the connection between generating functions and Klain--Schneider functions.

	\section{Mixed spherical projections and liftings}
	\label{sec:proj_lift}

	In this section, we recall weighted spherical projections and lifting and introduce a mixed version of them with the main goal of proving \cref{mixed_area_meas_lift:intro}.	
	To that end, we need to fix some notation:
	For $E\in\Gr_{n,k}$,
	\begin{equation*}
		\H^{n-k}(E,u)
		= \{ v\in\S^{n-1}\setminus E^\perp: \pr_{\! E }v=u \}
	\end{equation*}
	denote the relatively open $(n-k)$-dimensional half-sphere generated by $E^\perp$ and $u\in\S^{k-1}(E)$.
	Note that $\H^{n-k}(E,u)= \S^{n-k}(E^\perp \vee u)\cap u^+$, where we denote by $E'\vee u=\mathrm{span}(E'\cup u)$ the subspace generated by $E'$ and $u\in\S^{n-1}$, and by $u^+=\{x\in\R^n: \langle x,u\rangle>0\}$ the open half-space in the direction of $u\in\S^{n-1}$. Throughout, integration on $j$-dimensional spheres or half-spheres, unless indicated otherwise, is with respect to the $j$-dimensional Hausdorff measure $\mathcal{H}^{j}$.
	
	Now we want to recall the weighted spherical liftings and projections that were introduced by Goodey, Kiderlen, and Weil~\cite{Goodey2011}. 
	These prove to be helpful when it comes to relating the geometry of convex bodies in a subspace to their geometry relative to the ambient space. In the following, $\mathcal{M}(\S^{n-1})$ denotes the space of finite signed measures on $\S^{n-1}$.
	
	\begin{defi}[{\cite{Goodey2011}}]\label{def:sph_proj_lift}
		Let $1\leq k\leq n$, $E\in\Gr_{n,k}$, and $m>-k$. The \emph{$m$-weighted spherical projection} is the bounded linear operator
		\begin{gather*}
			\pi_{E,m} : C(\S^{n-1}) \to C(\S^{k-1}(E)), \\
			(\pi_{E,m}f)(u)
			= \int_{\H^{n-k}(E,u)} f(v) \langle u,v\rangle^{k+m-1} dv,
			\qquad u\in\S^{k-1}(E).
		\end{gather*}
		The \emph{$m$-weighted spherical lifting} is its adjoint operator
		\begin{equation*}
			\pi_{E,m}^\ast: \mathcal{M}(\S^{k-1}(E)) \to \mathcal{M}(\S^{n-1}).
		\end{equation*}
	\end{defi}
	
	If $m>0$, then $(\pi_{E,m}^\ast f)(u) = \norm{u|E}^m f(\pr_{\! E } u)$ for $f\in C(\S^{k-1}(E))$ and $u\in\S^{n-1}$. Hence, $\pi_{E,m}^\ast$ restricts to a bounded linear operator $C(\S^{k-1}(E))\to C(\S^{n-1})$, and by duality, $\pi_{E,m}$ naturally extends to an operator $\mathcal{M}(\S^{n-1})\to\mathcal{M}(\S^{k-1}(E))$.

	Goodey, Kiderlen, and Weil~\cite{Goodey2011} showed the following formula for convex bodies of a subspace, expressing its $i$-th order area measure relative to the ambient space in terms of the one relative to the subspace. As a general reference for mixed area measures and volumes, we cite \cite[Chapters~4 and~5]{Schneider2014}.
	
	\begin{thm}[{\cite[Theorem~6.2]{Goodey2011}}] \label{area_meas_lift}
		Let $1\leq i< k< n$ and $E\in\Gr_{n,k}$. Then for all $K\in\K(E)$,
		\begin{equation} \label{eq:area_meas_lift}
			S_i(K,{}\cdot{})
			= \frac{\binom{k-1}{i}}{\binom{n-1}{i}} \pi_{E,-i}^\ast S_i^E(K,{}\cdot{}).
		\end{equation}
	\end{thm}
	
	For our purposes, the relevant instance of this is when $i=k-1$ and the measure $S_i^{E}(K,{}\cdot{})$ is the surface area measure of $K$ relative $E$. In this instance, we propose a generalization of \cref{def:sph_proj_lift}. Here and in the following, the notational convention $(L_1,\ldots,L_i) | E = (L_1|E,\ldots,L_i|E)$
	where $L_1,\ldots,L_i\in\K(\R^n)$ and $E\subseteq\R^n$ is a linear subspace, will be used frequently.
	
	\begin{defi}
		Let $1\leq k\leq n$ and $E\in\Gr_{n,k}$. Also, let $C_1,\ldots ,C_{n-k}\in\K(\R^n)$ and set $\CC=(C_1,\ldots,C_{n-k})$. The \emph{$\CC$-mixed spherical projection} is the bounded linear operator
		\begin{gather*}
			\pi_{E,\CC}: C(\S^{n-1}) \to C(\S^{k-1}(E)), \\
			(\pi_{E,\CC}f)(u)
			= \int_{\H^{n-k}(E,u)} f(v) ~S^{E^\perp \vee u}(\CC| (E^\perp \vee u), dv),
			\qquad u\in\S^{k-1}(E).
		\end{gather*}
		The \emph{$\CC$-mixed spherical lifting} is its adjoint operator
		\begin{equation*}
			\pi_{E,\CC}^\ast: \mathcal{M}(\S^{k-1}(E))\to\mathcal{M}(\S^{n-1}).
		\end{equation*} 
	\end{defi}
	
	In the instance where $k=n$ and the family $\CC$ is empty, the above is to be understood as $\pi_{E,\CC}=\Id$ and $\pi_{E,\CC}^\ast =\Id$.	
	Next, as an intermediate step, we prove a polytopal version of \cref{mixed_area_meas_lift:intro}. The following reduction theorem for mixed volumes will be a key ingredient.
	
	\begin{thm}[{\cite[Theorem 5.3.1]{Schneider2014}}]\label{mixed_vol_reduction}
		Let $1\leq k< n$ and $E\in \Gr_{n,k}$. For all convex bodies $K\in\K(E)$ and $C_1,\ldots,C_{n-k}\in \K(\R^n)$,
		\begin{equation}\label{eq:mixed_vol_reduction}
			V(K^{[k]},C_1,\ldots,C_{n-k})
			= \frac{1}{\binom{n}{k}} V_k(K)V^{E^\perp}\!(C_1 | E^\perp,\ldots,C_{n-k} | E^\perp),
		\end{equation}
		where $V^{E^\perp}\!\!$ denotes the mixed volume relative to the subspace $E^{\perp}$.
	\end{thm}
	
	\begin{thm}\label{mixed_area_meas_lift_polytopes}
		Let $1\leq i< n-1$, $E\in\Gr_{n,i+1}$, and $\mathcal{Q}=(Q_1,\ldots,Q_{n-i-1})$ be a family of polytopes in $\R^n$. Then for every polytope $P\in\K(E)$,
		\begin{equation}\label{eq:mixed_area_meas_lift_polytopes}
			\mathbbm{1}_{\S^{n-1}\setminus E^\perp} S(P^{[i]},\mathcal{Q},{}\cdot{})
			= \frac{1}{\binom{n-1}{i}}\pi_{E,\mathcal{Q}}^\ast S_i^E(P,{}\cdot{}).
		\end{equation}
	\end{thm}
	\begin{proof}
		First, observe that both sides of \cref{eq:mixed_area_meas_lift_polytopes} are multilinear in the reference polytopes $Q_1,\ldots,Q_{n-i-1}$. Thus, by polarization, it suffices to consider the case where $Q_1=\cdots=Q_{n-i-1}=Q$ for some polytope $Q\in\K(\R^n)$.
		
		The mixed area measure of polytopes $P_1,\ldots,P_{n-1}\in\K(\R^n)$ can be written as
		\begin{equation}\label{eq:mixed_area_meas_polytopes}
			S(P_1,\ldots,P_{n-1},{}\cdot{})
			= \sum_{u\in\S^{n-1}} V^{u^\perp}\!(F(P_1,u),\ldots,F(P_{n-1},u))\delta_u,
		\end{equation}
		where $V^{u^\perp}\!$ denotes the mixed volume relative to $u^{\perp}$ and $F(P_j,u)$ denotes the \emph{support set} of $P_j$ with outer unit normal $u\in\S^{n-1}$ (cf.\ \cite[(5.22)]{Schneider2014}). Note that the sum in fact extends only over the outer unit normals of the facets of $P_1+\cdots+P_{n-1}$, and thus, is a finite sum.
		By~\cref{eq:mixed_area_meas_polytopes}, we have that for every $f\in C(\S^{n-1})$,
		\begin{align*}
			&\int_{\S^{n-1}\setminus E^\perp} \! f(u) S(P^{[i]},Q^{[n-i-1]},du)
			= \sum_{u\in\S^{n-1}\setminus E^\perp} \!\! f(u) V^{u^\perp}\!(F(P,u)^{[i]},F(Q,u)^{[n-i-1]}) \\
			&\qquad = \frac{1}{\binom{n-1}{i}} \sum_{u\in\S^{n-1}\setminus E^\perp} f(u) V_i(F(P,u))V_{n-i-1}(F(Q,u)| (E\cap u^\perp)^{\perp_{(u^\perp)}}),
		\end{align*}
		where in the second equality, we applied reduction formula~\cref{eq:mixed_vol_reduction} relative to~$u^\perp$, using the fact that $E\cap u^\perp \in\Gr_{n,i}^{u^\perp}$ for $u\in\S^{n-1}\setminus E^{\perp}$, and denoting by $(E\cap u^\perp)^{\perp_{(u^\perp)}}$ the orthogonal complement of $E\cap u^\perp$ relative to $u^\perp$.
		
		Every $x\in \R^n$ can be written as the orthogonal sum $x=x|(E\cap u^\perp) + x|(E^\perp\vee u)$;
		by choosing $x\in u^\perp$, this yields $x|(E\cap u^\perp)^{\perp_{(u^\perp)}}=x|(E^\perp\vee u)$.
		Thus,		
		\begin{align}\label{eq:mixed_area_meas_lift_polytopes:proof01}
			\begin{split}
				&\int_{\S^{n-1}\setminus E^\perp} \! f(u) S(P^{[i]},Q^{[n-i-1]},du) \\
				&\qquad = \frac{1}{\binom{n-1}{i}} \sum_{u\in\S^{n-1}\setminus E^\perp} f(u) V_i(F(P,u))V_{n-i-1}(F(Q,u)| E^\perp\vee u) \\
				&\qquad = \frac{1}{\binom{n-1}{i}} \sum_{u\in\S^i(E)} \! V_i(F(P,u)) \!\! \sum_{v\in\H^{n-i-1}(E,u)} \!\!\! f(v) V_{n-i-1}(F(Q,v)| (E^\perp \vee u)),
			\end{split}			
		\end{align}
		where in the second equality, we used that for every $v\in \S^{n-1}\setminus E^\perp$, there exists a unique $u\in\S^i(E)$ such that $v\in \H^{n-i-1}(E,u)$; moreover, $F(P,v)=F(P,u)$ and $E^\perp\vee v=E^\perp\vee u$.
		Next, note that an instance of \cref{eq:mixed_area_meas_polytopes},
		\begin{equation}\label{eq:mixed_area_meas_lift_polytopes:proof02}
			(\pi_{E,Q}f)(u)
			= \sum_{v\in\H^{n-i-1}(E,u)} f(v) V_{n-i-1}(F(Q | (E^\perp \vee u),v)).
		\end{equation}
		Note that $F(C,v)|E'=F(C|E',v)$ for every convex body $C\in\K(\R^n)$, subspace $E'\subseteq\R^n$, and direction $v\in E'$. Hence, plugging \cref{eq:mixed_area_meas_lift_polytopes:proof02} into \cref{eq:mixed_area_meas_lift_polytopes:proof01} yields
		\begin{align*}
			&\int_{\S^{n-1}\setminus E^\perp} \! f(u) S(P^{[i]},Q^{[n-i-1]},du)
			= \frac{1}{\binom{n-1}{i}} \sum_{u\in\S^i(E)} (\pi_{E,Q}f)(u) V_i(F(P,u))\\
			&\qquad = \frac{1}{\binom{n-1}{i}} \int_{\S^i(E)} (\pi_{E,Q}f)(u) S_i^E(P,du)
			= \frac{1}{\binom{n-1}{i}} \int_{\S^{n-1}}  f(u) ~[\pi_{E,Q}^{\ast}S_i^E(P,{}\cdot{})](du),
		\end{align*}
		where the second equality is another instance of \cref{eq:mixed_area_meas_polytopes}. Since $f\in C(\S^{n-1})$ was arbitrary, this yields $\mathbbm{1}_{\S^{n-1}\setminus E^\perp}S(P^{[i]},Q^{[n-i-1]},{}\cdot{})=1/\binom{n-1}{i}\pi_{E,Q}^\ast S_i^E(P,{}\cdot{})$. As was noted at the beginning of the proof, this shows the theorem.
	\end{proof}
	
	In the theorem above, we have deliberately avoided the set $\S^{n-1}\cap E^\perp$. This is because in the polytopal case, $S(P^{[i]},\mathcal{Q},{}\cdot{})$ assigns positive mass to this set which can not be captured by the mixed spherical lifting. However, if we replace the polytopes~$Q_j$ by smooth bodies~$C_j$, then this is no longer the case, as the following lemma shows.
	
	\begin{lem}\label{mixed_area_meas_abs_cont}
		Let $0\leq i< n-1$ and $\CC=(C_1,\ldots,C_{n-i-1})$ be a family of convex bodies with $C^2$~support functions. Then for all $K\in\K(\R^n)$, the mixed area measure $S(K^{[i]},\CC,{}\cdot{})$ is absolutely continuous with respect to $\mathcal{H}^{n-i-1}$.
	\end{lem}
	\begin{proof}
		First, note that the $i$-th area measure of $K$ is absolutely continuous with respect to $\mathcal{H}^{n-i-1}$ (cf.\ \cite[Theorem~4.5.5]{Schneider2014}). In order to pass to the mixed area measure, recall that if $L_1,\ldots,L_{n-1}\in\K(\R^n)$ have $C^2$~support functions, then
		\begin{equation*}
			S(L_1,\ldots,L_{n-1},du)
			= \mathsf{D}(D^2h_{L_1}(u),\ldots,D^2h_{L_{n-1}}(u))du,
		\end{equation*}
		where $D^2h_L(u)$ denotes the restriction of the Hessian of the support function $h_L$ (as a one-homogeneous function on $\R^n$) to the hyperplane $u^\perp$ and $\mathsf{D}$ denotes the mixed discriminant (cf.\ \cite[(2.68) and (5.48)]{Schneider2014}). Thus, whenever $K$ also has a $C^2$~support function,
		\begin{align*}
			&S(K^{[i]},C_1,\ldots,C_{n-i-1},du)
			= \mathsf{D}(D^2h_K(u)^{[i]},D^2h_{C_1}(u),\ldots,D^2h_{C_{n-i-1}}(u))du \\
			&\qquad \leq \norm{h_{C_1}}_{C^2}\cdots \norm{h_{C_{n-i-1}}}_{C^2} \mathsf{D}(D^2h_K(u)^{[i]},\Id^{[n-i-1]})du
			= M S_i(K,du),
		\end{align*}
		where we used the monotonicity of mixed discriminants and $M>0$ is a constant depending on $C_1,\ldots,C_{n-i-1}$ but not on $K$.
		The obtained inequality
		\begin{equation*}
			S(K^{[i]},C_1,\ldots,C_{n-i-1},{}\cdot{})
			\leq M S_i(K,{}\cdot{}),
		\end{equation*}
		by continuity of the mixed area measure, extends to all convex bodies $K\in\K(\R^n)$, which concludes the proof.  
	\end{proof}
	
	Now we want to pass from the polytopal to the smooth case. To that end, we need the following formulation of the classical Portmanteau theorem that characterizes weak convergence of measures.

	\begin{thm}[{\cite[Theorem~13.16]{Klenke2020}}]\label{Portmanteau}
		Let $\mu_m,\mu$ be finite positive measures on a compact metric space $X$. Then the following are equivalent:
		\begin{enumerate}[label=\upshape(\alph*)]
			\item \label{Portmanteau:weak_conv}
			$\mu_m \to \mu$ weakly.
			\item \label{Portmanteau:cont_fct}
			For every $f\in C(X)$, we have $\lim_{m\to\infty} \int_X f d\mu_m	= \int_X f d\mu$.
			\item \label{Portmanteau:bounded_fct}
			For every bounded, measurable function $f$ on $X$ such that its discontinuity points are a set of $\mu$-measure zero, $\lim_{m\to\infty} \int_X f d\mu_m	= \int_X f d\mu$.
		\end{enumerate}
	\end{thm}
	
	We are now ready to prove \cref{mixed_area_meas_lift:intro}, which we state here again.
	
	\begin{thm}\label{mixed_area_meas_lift}
		Let $0\leq i< n-1$, $E\in\Gr_{n,i+1}$, and $\CC=(C_1,\ldots,C_{n-i-1})$ be a family of convex bodies with $C^2$~support functions. Then for all $K\in\K(E)$,
		\begin{equation}\label{eq:mixed_area_meas_lift}
			S(K^{[i]},\CC,{}\cdot{})
			= \frac{1}{\binom{n-1}{i}}\pi_{E,\CC}^\ast S_i^E(K,{}\cdot{}).
		\end{equation}
	\end{thm}
	\begin{proof}
		As in the proof of \cref{mixed_area_meas_lift_polytopes}, it suffices to consider the case where \linebreak $C_1=\cdots=C_{n-i-1}=C$ for some body $C\in\K(\R^n)$ with a $C^2$~support function. Next, let $K=P\in\K(E)$ be a polytope and let $Q_m\in\K(\R^n)$ be a sequence of polytopes such that $Q_m\to C$ in the Hausdorff metric. By \cref{mixed_area_meas_lift_polytopes}, for every $f\in C(\S^{n-1})$,		
		\begin{equation*}
			\int_{\S^{n-1}} \! \mathbbm{1}_{\S^{n-1}\setminus E^\perp}\!(u) f(u) S(P^{[i]},Q_m^{[n-i-1]},du)
			= \frac{1}{\binom{n-1}{i}} \int_{\S^i(E)}\! (\pi_{E,Q_m}f)(u) S^E_i(P,du).
		\end{equation*}
		
		Now we want to pass to the limit $m\to \infty$ on both sides. On the left hand side, we have weak convergence of the mixed area measures and according to \cref{mixed_area_meas_abs_cont}, the mixed area measure $S(P^{[i]},C^{[n-i-1]},{}\cdot{})$ vanishes on $\S^{n-1}\cap E^\perp$, as it is of Hausdorff dimension $n-i-2$. Thus, by \cref{Portmanteau},
		\begin{equation*}
			\lim_{m\to\infty}\int_{\S^{n-1}}\! \mathbbm{1}_{\S^{n-1}\setminus E^\perp}\!(u) f(u) S(P^{[i]},Q_m^{[n-i-1]},du)
			= \int_{\S^{n-1}} \! f(u) S(P^{[i]},C^{[n-i-1]},du).
		\end{equation*}
		
		For the limit on the right hand side, note that for every fixed $u\in\S^i(E)$, the projections $Q_m| (E^\perp\vee u)$ converge to $C| (E^\perp\vee u)$. This implies weak convergence of the respective surface area measures relative to $E^\perp\vee u$. The support function of $C| (E^\perp\vee u)$ is just the restriction to~$\S^{n-i-1}(E^\perp\vee u)$ of the support function of $C$, and thus, is again of class~$C^2$. Consequently, the respective surface area measure relative to $E^\perp\vee u$ is absolutely continuous; in particular, it vanishes on $\S^{n-i-1}(E^\perp\vee u)\cap E^\perp$, so \cref{Portmanteau} implies that
		\begin{align*}
			&\lim_{m\to\infty} \int_{\S^{n-i-1}(E^\perp\vee u)} \mathbbm{1}_{\H^{n-i-1}(E,u)}(v) f(v) S_{n-i-1}^{E^\perp \vee u}(Q_m| (E^\perp \vee u),dv)\\
			&\qquad = \int_{\S^{n-i-1}(E^\perp\vee u)} \mathbbm{1}_{\H^{n-i-1}(E,u)}(v) f(v) S_{n-i-1}^{E^\perp \vee u}(C| (E^\perp \vee u),dv).
		\end{align*}
		That is, $\lim_{m\to\infty} \pi_{E,Q_m}f = \pi_{E,C}f$ pointwise on $\S^i(E)$. Since the measure $S_i^{E}(P,{}\cdot{})$ has finite support, passing to the limit $m\to\infty$ yields
		\begin{equation*}
			\int_{\S^{n-1}} f(u) S(P^{[i]},C^{[n-i-1]},du)
			= \frac{1}{\binom{n-1}{i}}\int_{\S^i(E)} (\pi_{E,C}f)(u) S^E_i(P,du).
		\end{equation*}
		Observe that both sides of this identity depend continuously on $P$, and thus, by approximation, we may replace $P$ by any convex body $K\in\K(\R^n)$. Since $f\in C(\S^{n-1})$ was arbitrary, we obtain for every $K\in\K(\R^n)$,
		\begin{equation*}
			S(K^{[i]},C^{[n-i-1]},{}\cdot{})
			= \frac{1}{\binom{n-1}{i}}\pi_{E,C}^\ast S_i^E(K,{}\cdot{}).
		\end{equation*}
		As was noted at the beginning of the proof, this shows the theorem.
	\end{proof}
	
	\begin{rem}
		Comparing \cref{area_meas_lift,mixed_area_meas_lift}, it might be natural to ask for a unification that deals with mixed area measures of the form $S(K^{[i]},{B^n}^{[k-i-1]},\CC,{}\cdot{})$, where $E\in\Gr_{n,k}$, $K\in\K(E)$, and $\CC=(C_1,\ldots,C_{n-k})$. However, the methods in the proofs of \cref{area_meas_lift,mixed_area_meas_lift} seem to be disconnected, and thus, it is not immediately clear to the authors how such a unification would look like and how it could be obtained.
	\end{rem}

	\section{Action on the Klain--Schneider function}
	\label{sec:LL_Lambda_KS}

	This section is devoted to describing the action of the Lefschetz operators on the Klain--Schneider function, which from now on we require to be centered. First, we introduce the aforementioned Radon type transforms. Then we consider the operator $\Lambda$, where we apply our findings from the previous section; afterwards we turn to the operator~$\LL$.

	\subsection{Two Radon type transforms}
	\label{sec:Radon}
	
	Let us now formally introduce the integral transforms that occur in \cref{LL_Lambda_KS:intro}. Throughout, integration on compact Grassmann manifolds, unless indicated otherwise, is with respect to the unique rotationally invariant probability measure.
	
	\begin{defi}
		Let $1\leq k\leq n$.
		\begin{enumerate}[label=\upshape(\roman*)]
			\item If $k>1$, we define $\RT_{k,k-1}: C(\Fl_{n,k})\to C(\Fl_{n,k-1})$ by
			\begin{equation*}
				[\RT_{k,k-1}\zeta](E,u)
				= \int_{\Gr_{n,k-1}^{E\cap u^\perp}} \zeta(F\vee u,\pr_{\! F^\perp} u)~dF,
				\qquad (E,u)\in\Fl_{n,k-1}.
			\end{equation*}
			\item If $k< n$, we define $\RT_{k,k+1}: C(\Fl_{n,k})\to C(\Fl_{n,k+1})$ by
			\begin{gather*}
				[\RT_{k,k+1}\zeta](E,u)
				= \int_{\Gr_{n,k}^{E}} \zeta(F,\pr_{\! F} u)\norm{u|F}~dF,
				\qquad (E,u)\in\Fl_{n,k+1}.
			\end{gather*}
		\end{enumerate}
	\end{defi}
	
	Subsequently, we will need the fact that these transforms map linear functions to linear functions. To that end, recall that the orthogonal projection $\pi_1$ from $L^2(\S^{n-1})$ onto the space of linear functions is given by
	\begin{align*}
		(\pi_1 f)(u)
		= \frac{1}{\omega_n} \int_{\S^{n-1}} \langle u,v\rangle f(v)~dv,
		\qquad u\in\S^{n-1},
	\end{align*}
	where $\omega_k=k\kappa_k$ denotes the surface area of $\S^{k-1}$.
	For $E\in\Gr_{n,k}$, we denote by $\pi_1^{E}$ the respective orthogonal projection relative to~$E$. Moreover, we define an operator $\pi_1^{\langle k\rangle}$ on the space $C(\Fl_{n,k})$ by
	$[\pi_1^{\langle k\rangle }\zeta](E,u)=[\pi_1^{E}\zeta(E,{}\cdot{})](u)$
	for $\zeta\in C(\Fl_{n,k})$ and $(E,u)\in\Fl_{n,k}$. The interplay between the Radon type transforms defined above and linear functions is summarized in the following proposition, which we prove in \cref{sec:proofs}.	
	
	\begin{prop}\label{Radon_linear}
		Let $1\leq k\leq n$.
		\begin{enumerate}[label=\upshape(\roman*)]
			\item \label{Radon_down_linear}
			If $k>1$, then $\RT_{k,k-1}\pi_1^{\langle k\rangle } = \pi_1^{\langle k-1\rangle } \RT_{k,k-1}'$.
			\item \label{Radon_up_linear}
			If $k<n$, then $ \RT_{k,k+1}\pi_1^{\langle k\rangle} =\frac{k+1}{k} \pi_1^{\langle k+1\rangle}\RT_{k,k+1}$.
		\end{enumerate}
	\end{prop}
	
	\noindent Here, we also defined an auxiliary transform $\RT_{k,k-1}': C(\Fl_{n,k})\to C(\Fl_{n,k-1})$ by
	\begin{equation*}
		[\RT_{k,k-1}'\zeta](E,u)
		=  \frac{\omega_{n-k+1}\omega_{n-k+3}\omega_{k-1}}{\omega_{n-k+2}^2\omega_k} \! \int_{\Gr^E_{n,k}} \!\! [\pi^F_{E,1}\zeta(F,{}\cdot{})](u) dF,
		\quad (E,u)\in\Fl_{n,k-1},
	\end{equation*}
	where $\pi_{E,1}^{F}: C(\S^{k-1}(F))\to C(\S^{k-2}(E))$ denotes the $1$-weighted spherical projection relative to~$F$.

	\subsection{The Lefschetz derivation operator}
	\label{sec:Lambda_KS}

	Our strategy to prove \cref{LL_Lambda_KS:intro}~\ref{Lambda_KS:intro} is to reduce the general case to certain mixed volume valuations: For $1\leq i\leq n-1$, consider valuations of the form
	$V({}\cdot{}^{[i]},\CC,f)\in\Val_i$, defined by
	\begin{equation}\label{eq:mixed_vol_val}
		V(K^{[i]},\CC,f)
		= \int_{\S^{n-1}} f(u) ~S(K^{[i]},\CC,du),
		\qquad K\in\K(\R^n),
	\end{equation}	
	where $\CC=(C_1,\ldots,C_{n-i-1})$ is a family of convex bodies with $C^2$~support functions and $f\in C(\S^{n-1})$.
	
	The reduction will require an approximation argument. To that end, recall first that by a classical result of McMullen~\cite{McMullen1980},
	\begin{equation*}
		\Val
		= \bigoplus_{i=0}^n \Val_i.
	\end{equation*}
	The space $\Val_0$ is spanned by the Euler characteristic and, due to a famous characterization by Hadwiger~\cite{Hadwiger1957}, the space $\Val_n$ is spanned by the volume. 
	We endow the space $\Val$ with the norm $\norm{\varphi}=\max\{\abs{\varphi(K)}:K\subseteq B^n\}$.
	It is easy to see that this is a complete norm on each space $\Val_i$. 
	By virtue of the homogeneous decomposition, the space $\Val$ is a Banach space.
	
	Moreover, there is a natural representation of the general linear group $\mathrm{GL}(n)$ on the space $\Val$. For $g\in\mathrm{GL}(n)$ and $\varphi\in\Val$, we set
	\begin{equation*}
		(g\cdot \varphi)(K)
		= \varphi(g^{-1}(K)),
		\qquad K\in\K(\R^n).
	\end{equation*}
	Clearly, each space of valuations with a given degree and parity is $\mathrm{GL}(n)$-invariant. By Alesker's \emph{irreducibility theorem}~\cite{Alesker2001}, these are the building blocks of all closed $\mathrm{GL}(n)$-invariant subspaces of $\Val$:
	
	\begin{thm}[{\cite[Theorem~1.3]{Alesker2001}}]\label{irreducibility}
		For $0\leq i\leq n$, the natural representation of $\mathrm{GL}(n)$ on the spaces $\Val_i^{\mathrm{even}}$ and $\Val_i^{\mathrm{odd}}$ is irreducible.
	\end{thm}
	
	This result has far-reaching consequences in valuation theory. Especially relevant for our purposes, it is not hard to see that for fixed $1\leq i\leq n-1$, the valuations of the form \cref{eq:mixed_vol_val} span a $\mathrm{GL}(n)$-invariant subspace of $\Val_i$ containing non-trivial even and odd elements.
	Therefore, by \cref{irreducibility}, the linear span of such mixed volumes is dense in $\Val_i$.
	
	As a direct consequence of \cref{mixed_area_meas_lift:intro}, the Klain--Schneider function of a valuation of the form \cref{eq:mixed_vol_val} is given by
	\begin{equation}\label{eq:KS_mixed_vol}
		\KS_{V({{}\cdot{}}^{[i]},\CC,f)}(E,u)
		= \frac{1}{\binom{n-1}{i}}\big[(\Id - \pi_1^E)(\pi_{E,\CC}f)\big](u),
		\qquad (E,u)\in\Fl_{n,i+1}.
	\end{equation}	
	
	In order to describe the action of the Lefschetz derivation operator $\Lambda$ on these Klain--Schneider functions, we need the following Cauchy--Kubota type formula for mixed area measures (cf.\ \cite[(4.78)]{Schneider2014}): For every family $\CC=(C_1,\ldots,C_{n-2})$ of convex bodies in $\R^n$ and all $f\in C(\S^{n-1})$,
	\begin{equation}\label{eq:Kubota_area_meas}		
		\int_{\S^{n-1}} f(u) S(\CC,B,du)
		= \frac{\omega_n}{\omega_{n-1}}\int_{\Gr_{n,n-1}} \int_{\S^{n-2}(H)} f(u) S^H\!(\CC | H, du)~ dH.
	\end{equation}
	
	We are now ready to prove \cref{LL_Lambda_KS:intro}~\ref{Lambda_KS:intro}.
	
	\begin{thm}\label{Lambda_KS}
		Let $1<i\leq n-1$ and $\varphi\in\Val_i$. Then
		\begin{equation}\label{eq:Lambda_KS}
			\KS_{\Lambda\varphi}
			= \frac{(n-i)\omega_{n-i+1}}{\omega_{n-i}} (\Id-\pi_1^{\langle i\rangle})\RT_{i+1,i}\KS_{\varphi}.
		\end{equation}
	\end{thm}
	\begin{proof}
		As a consequence of \cref{irreducibility}, valuations of the form
		$\varphi=V({{}\cdot{}}^{[i]},\CC,f)$, where $\CC=(C_1,\ldots,C_{n-i-1})$ is a family of convex bodies with $C^2$~support functions and $f\in C(\S^{n-1})$, span a dense subspace of $\Val_i$.
		Hence, by linearity and by continuity of the Klain--Schneider map and the operator~$\Lambda$, it suffices to consider these valuations.
		
		Due to the Steiner formula for mixed area measures, $\Lambda\varphi= iV({}\cdot{}^{[i-1]},\CC,B^n,f)$. Hence, by \cref{eq:KS_mixed_vol}, we have that $\KS_{\varphi}	= (\Id - \pi_1^{\langle i+1\rangle})\zeta_{\varphi}$ and $\KS_{\Lambda\varphi}= (\Id - \pi_1^{\langle i\rangle})\zeta_{\Lambda\varphi}$, where we defined 
		\begin{equation*}
			\zeta_{\varphi}(F,v)
			=\frac{1}{\binom{n-1}{i}}(\pi_{F,\CC}f)(v)
			\qquad\text{and}\qquad
			\zeta_{\Lambda\varphi}(E,u)
			= \frac{i}{\binom{n-1}{i-1}} (\pi_{E,(\CC,B^n)}f)(u)
		\end{equation*}
		for $(F,v)\in\Fl_{n,i+1}$ and $(E,u)\in\Fl_{n,i}$.		
		By applying the definition of $\pi_{E,(\CC,B^n)}$ and an instance of \cref{eq:Kubota_area_meas} relative to the space $E^\perp\vee u$, we obtain
		\begin{align*}
			& \zeta_{\Lambda\varphi}(E,u)
			= \frac{i}{\binom{n-1}{i-1}} \int_{\H^{n-i}(E,u)} \! f(v) S^{E^\perp\vee u}(\CC | (E^\perp\vee u), B^n|(E^\perp\vee u),dv) \\
			&\qquad = \frac{i}{\binom{n-1}{i-1}} \frac{\omega_{n-i+1}}{\omega_{n-i}} \int_{\Gr^{E\cap u^\perp}_{n,i}} \int_{\H^{n-i-1}(F\vee u,\pr_{\! F^\perp}u)} \! f(v) S^{F^\perp}\!(\CC | F^\perp,dv)~ dF,
		\end{align*}
		where we used that whenever $F\in \Gr^{E\cap u^\perp}_{n,i}$, then $(\CC|(E^\perp\vee u))|F^{\perp} = \CC|F^{\perp}$ and
		\begin{equation*}
			\H^{n-i}(E,u)\cap F^\perp
			= \H^{n-i-1}(F\vee u,\pr_{\! F^\perp}u).
		\end{equation*}
		Moreover, since $F^\perp=(F\vee u)^{\perp}\vee \pr_{\! F^\perp} u$ for almost all $F\in \Gr^{E\cap u^\perp}_{n,i}$,
		applying again the definition of the mixed spherical projection yields
		\begin{align*}
			& \zeta_{\Lambda\varphi}(E,u)
			=  \frac{i}{\binom{n-1}{i-1}} \frac{\omega_{n-i+1}}{\omega_{n-i}}  \int_{\Gr^{E\cap u^\perp}_{n,i}} \! (\pi_{F\vee u,\CC}f)(\pr_{\! F^\perp} u) dF \\
			&\qquad	= \frac{(n-i)\omega_{n-i+1}}{\omega_{n-i}} [\RT_{i+1,i}\zeta_{\varphi}](E,u).
		\end{align*}
		Finally, the linear components remain to be eliminated. Due to \cref{Radon_linear}~\ref{Radon_down_linear},
		\begin{equation*}
			\RT_{i+1,i}\KS_{\varphi}
			= \RT_{i+1,i}\zeta_{\varphi} - \RT_{i+1,i}\pi_1^{\langle i+1\rangle}\zeta_{\varphi}
			= \frac{\omega_{n-i}}{(n-i)\omega_{n-i+1}} \zeta_{\Lambda\varphi} - \pi_1^{\langle i\rangle}\RT_{i+1,i}' \zeta_{\varphi}.
		\end{equation*}
		By applying $\Id-\pi_1^{\langle i\rangle}$ to both sides of this equation, we obtain \cref{eq:Lambda_KS}.
	\end{proof}
	
	\begin{rem}\label{rem:generalized_Lambda}
		For an even measure $\mu\in\mathcal{M}(\S^{n-1})$, we define a generalized Lefschetz operator on $\Val$ by
		\begin{equation*}
			(\Lambda_\mu \varphi)(K)
			= \left.\frac{d}{dt}\right|_{t=0^+} \varphi(K+tZ_\mu) 
		\end{equation*}
		where $Z_\mu$ is the zonoid generated by the measure $\mu$ (cf.\ \cite[p.~193]{Schneider2014}). It is not too difficult to generalize the Cauchy--Kubota type formula~\cref{eq:Kubota_area_meas} to zonoids:
		\begin{equation*}		
			\int_{\S^{n-1}}f(v)S(\CC,Z_\mu,dv)
			= \frac{2}{n-1}\int_{\S^{n-1}} \int_{\S^{n-2}(u^\perp)} f(v) S^{u^\perp}\!(\CC | u^\perp, dv)~ \mu(du).
		\end{equation*}
		Note also that $Z_\mu | E' = Z_{\pi_{E',1}\mu}$ for every subspace $E'\subseteq \R^n$.  We introduce a generalized transform $\RT_{i+1,i}^{\mu}: C(\Fl_{n,i+1})\to C(\Fl_{n,i})$ by
		\begin{equation*}
			[\RT_{i+1,i}^{\mu}\zeta](E,u)
			= \int_{\Gr^{E\cap u^\perp}_{n,i}} \! \zeta(F\vee u,\pr_{\! F^\perp} u) ~[J^{(E,u)}_\ast(\pi_{E^\perp\vee u,1}\mu)](dF),
			\quad (E,u)\in \Fl_{n,i},
		\end{equation*}
		where we defined the map $J^{(E,u)}:\S^{n-i}(E^\perp\vee u)\to\Gr_{n,i}^{E\cap u^\perp}: w\mapsto (E\cap u^\perp)\vee w$. Then, by carrying out the same argument as in the proof of \cref{Lambda_KS}, we obtain that $\KS_{\Lambda_{\mu}\varphi}(E,u)
		= 2 (\Id-\pi_1^{\langle i\rangle})\RT_{i+1,i}^{\mu}\KS_{\varphi}$.
		
		This has the following application: Every even, smooth valuation $\psi\in\Val_{n-1}$ (see \cref{sec:Val_sph}) admits a representation
		\begin{equation*}
			\psi(K)
			= \int_{\S^{n-1}} V_{n-1}(K| u^\perp) \mu(du),
			\qquad K\in\K(\R^n),
		\end{equation*}		
		where $\mu$ is a signed measure on $\S^{n-1}$ with a smooth density, called \emph{Crofton measure} of $\psi$. Then for every smooth valuation $\varphi\in\Val$, its \emph{Bernig--Fu convolution} (see~\cite{Alesker2017}) with $\psi$ is given by $\varphi\ast\psi= (\Lambda_{\mu^+}-\Lambda_{\mu^-})\varphi$,
		where $\mu=\mu^+ - \mu^-$ is the Jordan decomposition of $\mu$. In this sense, we obtain a description of the Bernig--Fu convolution of a smooth valuation with an even, smooth valuation of codegree one.
	\end{rem}

	\subsection{The Lefschetz integration operator}
	\label{sec:LL_KS}

	We now turn to the action of the operator~$\LL$ on the Klain--Schneider function, proving \cref{LL_Lambda_KS:intro}~\ref{LL_KS:intro}. We want to stress that the proof utilizes only basic tools from convex and integral geometry.
	
	On affine Grassmann manifolds, there exists a positive rigid motion invariant Radon measure which is unique up to normalization. For $0\leq k\leq n$, we fix this normalization by setting
	\begin{equation*}
		\int_{\AGr_{n,k}} f(E)dE
		= \int_{\Gr_{n,k}} \int_{E'^\perp} f(E'+x)dx~dE'
	\end{equation*}
	for $f\in C_c(\AGr_{n,k})$. Below, we will need the following change of variables rule for affine Grassmannians:
	If $0\leq j\leq k\leq n$ and $E\in\AGr_{n,k}$, then
	\begin{equation}\label{eq:AGr_reparametrizaton}
		\int_{\AGr_{n,n-j}} \!\! f(E\cap F) dF
		= \frac{\flag{n-j}{k-j}}{\flag{n}{k}} \int_{\AGr_{n,k-j}^E} \!\! f(F)dF
	\end{equation}
	for all $f\in C_c(\AGr_{n,k-j}^E)$, where $\AGr_{n,k-j}^E$ denotes the Grassmann manifold of affine $(k-j)$-dimensional affine subspaces of~$E$ and $\flag{n}{k}=\binom{n}{k}\frac{\kappa_n}{\kappa_k\kappa_{n-k}}$, the \emph{flag coefficient}.
	
	Now we prove \cref{LL_Lambda_KS:intro}~\ref{LL_KS:intro}.
	
	\begin{thm}\label{LL_KS}
		Let $1\leq i< n-1$ and $\varphi\in\Val_i$. Then
		\begin{equation}\label{eq:LL_KS}
			\KS_{\LL\varphi}
			= \frac{\flag{n-1}{i+1}}{\flag{n}{i+2}} \RT_{i+1,i+2}\KS_{\varphi}.
		\end{equation}
	\end{thm}
	\begin{proof}
		First, we fix a subspace $E\in\Gr_{n,i+2}$. We will show \cref{eq:LL_KS} by evaluating $\LL\varphi$ on arbitrary polytopes $P\in\K(E)$. By the definition of the Lefschetz integration operator~$\LL$ and~\cref{eq:AGr_reparametrizaton},
		\begin{align}\label{eq:LL_KS:proof}
			\begin{split}
				&(\LL\varphi)(P)
				= \int_{\AGr_{n,n-1}} \!\! \varphi(P\cap (E\cap H)) dH
				= \frac{\flag{n-1}{i+1}}{\flag{n}{i+2}} \int_{\AGr_{n,i+1}^{E}} \!\! \varphi(P\cap F) dF \\
				&\qquad = \frac{\flag{n-1}{i+1}}{\flag{n}{i+2}} \int_{\Gr_{n,i+1}^E} \int_{E\cap H'^\perp} \varphi(P\cap(H'+x)) dx ~dH'.
			\end{split}
		\end{align}
		Next, denote by $F_j$ the facets of $P$ relative to $E$ with corresponding outer unit normals $u_j \in\S^{i+1}(E)$, where $j\in\{1,\ldots,N\}$. Observe that for $H'\in\Gr_{n,i+1}^E$ and $x\in E\cap H'^\perp$, the outer unit normals and facets of $P \cap (H'+x)$ relative to $H'+x$ are given by the vectors $\pr_{\! H'} u_j$ and sets $F_j\cap (H'+x)$ whenever they are $i$-dimensional. Thus, by the translation invariance of $\varphi$,
		\begin{align*}
			&\varphi(P\cap (H'+x))
			= \int_{\S^i(H')} \KS_\varphi(H',u) S_i^{H'}((P-x)\cap H',du)\\
			&\qquad = \sum_{j=1}^{N} \KS_\varphi(H',\pr_{\! H'}u_j) V_i(F_j \cap (H'+x)).
		\end{align*}
		Note that for every $u\in\S^{i+1}(E)$ and convex body $L \in\K(E\cap u^\perp)$,
		\begin{equation*}
			\int_{E\cap H'^\perp} V_i(L \cap (H'+x)) dx
			= \norm{u|H'}V_{i+1}(L),
		\end{equation*}
		as follows from Fubini's theorem.
		By combining these identities, it follows that
		\begin{align*}
			& \int_{E\cap H'^\perp} \varphi(P\cap(H'+x)) dx
			= \sum_{j=1}^{N} \KS_\varphi(H',\pr_{\! H'}u_j) \norm{u_j|H'}V_{i+1}(F_j) \\
			&\qquad = \int_{\S^{i+1}(E)} \KS_{\varphi}(H',\pr_{\! H'} u) \norm{u|H'} S^E_{i+1}(P,du).
		\end{align*}
		Plugging this into \cref{eq:LL_KS:proof} and changing the order of integration yields
		\begin{equation*}
			(\LL\varphi)(P)
			= \frac{\flag{n-1}{i+1}}{\flag{n}{i+2}} \int_{\S^{i+1}(E)} \int_{\Gr_{n,i+1}^E}  \!\!\! \KS_{\varphi}(H',\pr_{\! H'} u) \norm{u|H'} dH'~S^E_{i+1}(P,du).
		\end{equation*}
		By the definition of the Klain--Schneider function, this proves \cref{eq:LL_KS} up to the addition of a linear function. By \cref{Radon_linear}~\ref{Radon_up_linear} however, $\RT_{i+1,i+2}\KS_{\varphi}$ is centered, and thus, we obtain \cref{eq:LL_KS}.		
	\end{proof}
	
	\begin{rem}\label{rem:generalized_LL}
		For a translation invariant signed Radon measure $\mu$ on $\AGr_{n,n-1}$, define a generalized Lefschetz operator on $\Val$ by
		\begin{equation*}
			(\LL_\mu \varphi)(K)
			= \int_{\Gr_{n,n-1}} \varphi(K\cap H) ~\mu(dH),
			\qquad K\in\K(\R^n).
		\end{equation*}
		We define a generalized transform $\RT_{i+1,i+2}^{\mu}:C(\Fl_{n,i+1})\to C(\Fl_{n,i+2})$ by
		\begin{equation*}
			[\RT_{i+1,i+2}^{\mu}\zeta](E,u)
			= \int_{\Gr_{n,i+1}^E} \!\! \zeta(F,\pr_{\! F} u)\norm{u | F}~\mu^E(dF),
			\qquad (E,u)\in\Fl_{n,i+2},
		\end{equation*}
		where $\mu^E$ is the unique signed measure on $\Gr_{n,i+1}^E$ such that for all $f\in C_c(\AGr_{n,i+1}^E)$,
		\begin{equation*}
			\int_{\AGr_{n,n-1}} \!\! f(E\cap H) \mu(dH)
			= \int_{\Gr_{n,i+1}^E} \int_{E \cap H'^\perp} f(H'+x)dx ~\mu^E(dH').
		\end{equation*}
		Then the argument in the proof of \cref{eq:LL_KS} shows that $\KS_{\LL_\mu\varphi}= \RT_{i+1,i+2}^{\mu}\KS_{\varphi}$.
		
		This has the following application: Every even, smooth valuation $\psi\in\Val_1$ (see \cref{sec:Val_sph}) admits a representation
		\begin{equation*}
			\psi(K)
			= \int_{\AGr_{n,n-1}} \!\! \chi(K\cap H) \mu(dH),
		\end{equation*}
		where $\chi$ denotes the Euler characteristic and $\mu$ is a translation invariant signed Radon measure with a smooth density on $\AGr_{n,n-1}$, called a \emph{Crofton measure} of $\psi$. Then for every valuation $\varphi\in\Val$, its \emph{Alesker product} (see \cite{Alesker2012}) with $\psi$ is given by $\varphi\cdot \psi=\LL_\mu\varphi$. In this sense, we obtain a description of the Alesker product of a continuous valuation with an even, smooth valuation of degree one.
	\end{rem}

	\section{Minkowski valuations}
	\label{sec:MSOs}

	In this section, we turn to Minkowski valuations and more specifically, the \emph{mean section operators}. This family of geometric operators, which was introduced by Goodey and Weil~\cite{Goodey1992}, are of particular interest to us as they will play a crucial role in the proof of \cref{rho_description:intro}.  
	
	\begin{defi}[{\cite{Goodey1992}}]
		For $0\leq j\leq n$, the \emph{$j$-th mean section body} $\MSO_j K$ of a convex body $K\in\K(\R^n)$ is defined by
		\begin{equation*}
			h_{\MSO_j K}(u)
			= \int_{\AGr_{n,j}} h_{K\cap E}(u) ~dE,
			\qquad u\in\S^{n-1}.
		\end{equation*}
	\end{defi}

	The $j$-th mean section operator $\MSO_j$ is a continuous, rotationally equivariant Minkowski valuation, but it is not translation invariant. By defining the $j$-th \emph{centered} mean section operator $\tilde{\MSO}_j$ as
	\begin{equation*}
		\tilde{\MSO}_j K
		= \MSO_j (K-s(K)),
	\end{equation*}
	where $s(K)$ is the Steiner point of $K$ (cf.\ \cite[p.~50]{Schneider2014}), it follows that $\tilde{\MSO}_j \in\MVal_i$ for $2\leq j\leq n$, where $i+j=n+1$. Determining their generating functions is highly non-trivial and requires the family of functions introduced by Berg~\cite{Berg1969} in his solution to the Christoffel problem~\cite{Christoffel1865}. The \emph{Berg functions} $g_j \in C^\infty(-1,1)$, for $j\geq 2$, are defined recursively as follows:
	\begin{align}\label{eq:Berg_fcts}
		\begin{split}
			g_2(t)
			&= \frac{1}{2\pi} (\pi-\arccos{t})(1-t^2)^{\frac{1}{2}} - \frac{1}{4\pi}t, \\
			g_3(t)
			&= \frac{1}{2\pi}\left(1 + t \log(1-t) + \left( \frac 43 - \log 2\right)t \right), \\
			g_{j+2}(t)
			&= \frac{j+1}{2\pi} g_j(t) + \frac{j+1}{2\pi(j-1)} t g_j'(t) +  \frac{j+1}{2\pi\omega_j}t,
		\end{split}			
	\end{align}
	where $\omega_j=j\kappa_j$ is the $(j-1)$-dimensional Hausdorff measure of ${\S}^{j-1}$.
	Goodey and Weil~\cite{Goodey2014} showed that essentially, the $j$-th Berg function generates the $j$-th mean section operator -- independently of the dimension $n\geq 3$ of the ambient space. 
	
	\begin{thm}[{\cite{Goodey2014}}]\label{MSO_Berg_fct}
		For $n\geq 3$ and $2\leq j\leq n$,
		\begin{equation}\label{eq:MSO_Berg_fct}
			f_{\tilde{\MSO}_j}
			= m_{n,j}(\Id-\pi_1) {g}_j(\langle e_n,{}\cdot{}\rangle).
		\end{equation}
	\end{thm}	
	
	Here, $m_{n,j}$ is a constant which depends only on $n$ and $j$, and was explicitly determined in~\cite{Goodey2014}. Moreover, $\pi_1 f$ denotes the linear component of a function $f\in L^1(\S^{n-1})$ (see \cref{sec:Radon}).

	\subsection{The Berg functions}
	\label{sec:Berg_fcts}

	Next, we discuss analytic properties of the Berg functions. Berg~\cite{Berg1969} showed that they satisfy the following second order differential equation:
	For every $j\geq 2$ and $t\in (-1,1)$,
	\begin{equation}\label{eq:Berg_fct_ODE}
		\frac{1}{j-1}(1-t^2)g_j''(t) - tg_j'(t) + g_j(t)
		= -\frac{j}{\omega_j} t.
	\end{equation}
	Moreover, $g_2$ extends to a continuous function on $[-1,1]$ and for all $j\geq 3$, the function $g_j$ extends to a continuous function on $[-1,1)$ and $\lim_{t\to 1}g_j(t)=-\infty$. Berg showed that for every $\varepsilon>0$,
	\begin{equation}\label{eq:g_t->+-1_pre}
		\lim_{t\to\pm 1}(1-t^2)^{\frac{j-3+\varepsilon}{2}}g_j(t)
		=0.
	\end{equation}
	
	In the following, we investigate the behavior of the Berg functions at the end points in more detail. In particular, we determine the precise order of the singularity at $t=1$ as well as the scaling limit.
	
	\begin{lem}\label{g_t->+-1}
		For every $j\geq 2$ and $m\geq 0$ where $(j,m)\notin \{(2,0),(3,0)\}$,
		\begin{align}
			\begin{split}\label{eq:g_t->+-1}
				\lim_{t\to -1} (1-t^2)^{\frac{j-3}{2} + m}g_j^{(m)}(t)
				&= 0, \\
				\lim_{t\to +1} (1-t^2)^{\frac{j-3}{2} + m}g_j^{(m)}(t)
				&= -(j-1)2^{m-2}\frac{\Gamma\big(\frac{j-3}{2} + m\big)}{\pi^{\frac{j-1}{2}}}.
			\end{split}			
		\end{align}
	\end{lem}
	\begin{proof}
		We show the lemma by induction on $j\geq 2$. First, note that for $j\in\{2,3\}$ and $m=1$, one can show \cref{eq:g_t->+-1} by a simple direct computation. In order to apply L'Hôpital's rule, we need to justify the existence of the limits for the higher order derivatives. To that end, observe that we can express them as
		\begin{gather*}
			g_2^{(m)}(t)
			= p_{2,m}(t)(1-t^2)^{\frac{1}{2} - m} + q_{2,m}(t)(1-t^2)^{\frac{1}{2} - m}\arccos t + r_{2,m}(t), \\
			g_3^{(m)}(t)
			= p_{3,m}(t)\log(1-t) + q_{3,m}(t)(1-t)^{-m} + r_{3,m}(t),
		\end{gather*}
		with some polynomials $p_{j,m}$, $q_{j,m}$ and $r_{j,m}$, where $j\in\{2,3\}$ and $m\geq0$. This can be shown by a simple induction on $m\geq 0$. Consequently, the limits in \cref{eq:g_t->+-1} exist for $j\in\{2,3\}$ and $m\geq 1$, and by applying L'Hôpital's rule repeatedly, we obtain the desired identities, which proves the lemma in the instance where $j\in\{2,3\}$.
		
		For the inductive step, suppose that \cref{eq:g_t->+-1} holds for some fixed $j\geq 2$ and all $m\geq 0$. Next, observe that $m$-fold differentiation of the recurrence relation in \cref{eq:Berg_fcts} yields the following recurrence relation for the $m$-th order derivatives:
		\begin{equation*}
			\frac{2\pi}{j+1}g_{j+2}^{(m)}(t)
			= \frac{j+m-1}{j-1}g_j^{(m)}(t) + \frac{1}{j-1}tg_j^{(m+1)}(t) + \frac{1}{\omega_j} \frac{d^m t}{dt^m},
		\end{equation*}
		as can be shown by a simple induction on $m\geq 0$.
		Now we multiply both sides with $(1-t^2)^{\frac{j-1}{2} +m }$, pass to the limit $t\to\pm 1$, and employ both the induction hypothesis and \cref{eq:g_t->+-1_pre}. In this way, we obtain the desired identities for $j+2$ and all $m\geq 0$, which concludes the argument.
	\end{proof}
	
	As was mentioned before, the functions $g_j$ extend continuously to $t=-1$. In order to investigate the behavior at $t=-1$ further, the change of variables $t=-\cos\theta$ turns out to be quite useful. That is, we define a family of functions~$\hat{g}_j$, where $j\geq 2$, by
	\begin{equation*}
		\hat{g}_j(\theta)
		= g_j(-\cos\theta),
		\qquad \theta\in(-\pi,\pi)\backslash\{0\}.
	\end{equation*}
	Transforming the recursion \cref{eq:Berg_fcts} yields the following recursion for the functions~$\hat{g}_j$:
	\begin{align}\label{eq:Berg_fcts_arc}
		\begin{split}
			\hat{g}_2(\theta)
			&= \frac{1}{2\pi} \theta\sin\theta + \frac{1}{4\pi}\cos\theta, \\
			\hat{g}_3(\theta)
			&= \frac{1}{2\pi}\left(1 - \cos\theta \log(1+\cos\theta) - \left( \frac 43 - \log 2\right)\cos\theta \right), \\
			\hat{g}_{j+2}(\theta)
			&= \frac{j+1}{2\pi} \hat{g}_j(\theta) - \frac{j+1}{2\pi(j-1)} \frac{\hat{g}_j'(\theta)}{\tan\theta} - \frac{j+1}{2\pi\omega_j}\cos\theta.
		\end{split}
	\end{align}
	
	This recursion shows that the functions $\hat{g}_j$ extend to even analytic functions on $(-\pi,\pi)$.	
	For $\hat{g}_2$ and $\hat{g}_3$, this is obvious. If $\hat{g}_j$ extends to an even analytic function on $(-\pi,\pi)$, then its derivative $\hat{g}_j'$ has a simple zero in $\theta=0$ which cancels with the zero of the tangent function, and thus, $\hat{g}_{j+2}$ also extends to an even analytic function on $(-\pi,\pi)$.
	
	This observation allows us to obtain new information about the regularity of the generating functions of the mean section operators. To that end, denote by $d_{\S^{n-1}}(u,v):=\arccos\langle u,v\rangle$ the geodesic distance on $\S^{n-1}$. The following lemma is a classical fact from Riemannian geometry; for the convenience of the reader, we provide an elementary proof.
	
	\begin{lem}\label{dist_square_smooth}
		For $v\in\S^{n-1}$, the function $d_{\S^{n-1}}(v,{}\cdot{})^2$ is smooth on $\S^{n-1}\setminus\{-v\}$.
	\end{lem}
	\begin{proof}
		Clearly, the function $q : \S^{n-1}\times\S^{n-1}\to\R$, defined as
		\begin{equation*}
			q(u,v)
			= d_{\S^{n-1}}(u,v)^2
			= (\arccos \langle u,v\rangle )^2,
		\end{equation*}
		is smooth on the set of all pairs $(u,v)$ for which $\abs{\langle u,v\rangle}<1$. Next, observe that $\sin^2\theta$, as an even analytic function, can be written as a power series of $\theta^2$. Thus, there is a unique analytic function $f$ such that $\sin^2 \theta = f(\theta^2)$ for $\theta\in\R$. Since
		\begin{equation*}
			f'(0)
			= \lim_{\theta\to 0} \frac{f(\theta^2)}{\theta^2}
			= \lim_{\theta\to 0} \left(\frac{\sin\theta}{\theta}\right)^2
			= 1
			\neq 0,
		\end{equation*}
		there exists $\delta>0$ such that $f$ is invertible on $(-\delta,\delta)$ and its inverse $f^{-1}$ is also smooth. Consequently, whenever $d_{\S^{n-1}}(u,v)^2<\delta$, we can express $q(u,v)$ as
		\begin{equation*}
			q(u,v)
			= d_{\S^{n-1}}(u,v)^2
			= f^{-1}(\sin^2 d_{\S^{n-1}}(u,v))
			= f^{-1}(1-\langle u,v\rangle^2).
		\end{equation*}
		This shows that $q$ is smooth on the set of all pairs $(u,v)$ for which $d_{\S^{n-1}}(u,v)^2<\delta$. Hence, we have found two open sets covering the space $\{(u,v)\in\S^{n-1}:u\neq -v\}$ such that $q$ is smooth on each set, which proves the lemma.
	\end{proof}
	
	\begin{rem}
		Let $n\geq 2$ and $j\geq 2$ and consider the function $g_j(\langle e_n,{}\cdot{}\rangle)$ on $\S^{n-1}$. Clearly, this function is smooth on $\S^{n-1}\setminus\{\pm e_n\}$. Since $\hat{g}_j$ is even and analytic, there exists a smooth function $q_j$ such that $\hat{g}_j(\theta)=q_j(\theta^2)$, and thus,
		\begin{equation*}
			g_j(\langle e_n,u\rangle)
			= g_j(-\cos d_{\S^{n-1}}(-e_n,u))
			= q_j(d_{\S^{n-1}}(-e_n,u)^2).
		\end{equation*}
		Combined with \cref{dist_square_smooth}, this shows that $g_j(\langle e_n,{}\cdot{}\rangle)$ is a smooth function on $\S^{n-1}\setminus\{e_n\}$. Consequently, the generating functions of the mean section operators are smooth outside the north pole.
	\end{rem}
	
	Next, we turn to the value that is attained at the south pole. To that end, we apply the change of variables $t=-\cos\theta$ to the differential equation \cref{eq:Berg_fct_ODE}, which yields the following: For all $\theta\in (-\pi,\pi)\setminus \{0\}$,
	\begin{equation}\label{eq:Berg_fct_ODE_arc}
		\frac{1}{j-1} \hat{g}_{j}''(\theta) + \frac{j-2}{j-1} \frac{\hat{g}_{j}'(\theta)}{\tan\theta} + \hat{g}_{j}(\theta)
		= \frac{j}{\omega_{j}} \cos\theta.
	\end{equation}
	This allows us to deduce a recurrence relation for the value $\hat{g}_j(0)$, which is instrumental in the proof of the following lemma.
	
	\begin{lem}\label{g(-1)<0}
		Let $j\geq 3$. Then $\lim_{t\to-1} g_j(t)< 0$.
	\end{lem}
	\begin{proof}
		We show that $\hat{g}_j(0)<0$ by induction on $j\geq 3$.
		Direct computation shows that $\hat{g}_3(0)	= - \frac{1}{6\pi}$ and $\hat{g}_4(0)= - \frac{3}{2\pi^2}$.
		Letting $\theta\to0$ in \cref{eq:Berg_fct_ODE_arc} yields
		\begin{equation*}
			\hat{g}_{j}''(0) + \hat{g}_{j}(0)
			= \frac{j}{\omega_{j}}.
		\end{equation*} 
		By taking the limit $\theta\to 0$ in the recurrence relation in \cref{eq:Berg_fcts_arc} and combining this with the identity above, we obtain that
		\begin{equation*}
			\hat{g}_{j+2}(0)
			= \frac{j+1}{2\pi}\hat{g}_j(0) - \frac{j+1}{2\pi(j-1)}\hat{g}_j''(0) - \frac{j+1}{2\pi\omega_j}
			= \frac{j+1}{2\pi(j-1)}\bigg(j\hat{g}_j(0) - \frac{2j-1}{\omega_j} \bigg),
		\end{equation*}
		which shows that if $\hat{g}_j(0)<0$, then also $\hat{g}_{j+2}(0)<0$.
	\end{proof}

	\subsection{The space \texorpdfstring{$\MVal_i$}{MVali}}
	\label{sec:MVali}

	We want to make some remarks about the structure of the space $\MVal_i$, which we will come back to in the subsequent section. Before that, since it is convenient and will also be needed later on, we want to recall the natural action of $\SO(n)$ on (generalized) functions on the unit sphere.
	
	For a rotation $\vartheta\in\SO(n)$ and a smooth function $\phi\in C^{\infty}(\S^{n-1})$, we set $(\vartheta \phi)(u)=\phi(\vartheta^{-1} u)$ for $u\in\S^{n-1}$. For a distribution $\gamma\in C^{-\infty}(\S^{n-1})$, we set $\langle\phi, \vartheta\gamma \rangle_{\S^{n-1}}=\langle\vartheta^{-1}\phi,\gamma\rangle_{\S^{n-1}}$, where $\langle{}\cdot{},{}\cdot{}\rangle_{\S^{n-1}}$ denotes the natural pairing of a distribution and a smooth function. An $\SO(n-1)$-invariant (generalized) function on $\S^{n-1}$ is also called \emph{zonal}.
	This also encompasses continuous functions, Lebesgue integrable functions, and signed measures,
	by virtue of the chain of identifications
	\begin{equation}\label{eq:chain_of_identifications}
		C^{\infty}(\S^{n-1})
		\subseteq C(\S^{n-1})
		\subseteq L^2(\S^{n-1})
		\subseteq \mathcal{M}(\S^{n-1})
		\subseteq C^{-\infty}(\S^{n-1}).
	\end{equation}
	
	To date, a complete classification of continuous, translation invariant and rotationally equivariant Minkowski valuations is not known. In view of the integral representation~\cref{eq:gen_fct:intro}, this open problem reduces to a characterization of generating functions. It is not hard to see that the support function $h_L$ of a convex body of revolution $L\in\K(\R^n)$ generates a Minkowski valuation $\Phi\in\MVal_i$ in every degree $i\in\{1,\ldots,n-1\}$. However, not all Minkowski valuations in $\MVal_i$ are of this form, as is exemplified by the mean section operators.
	
	Since the space $\MVal_i$, endowed with the pointwise Minkowski operations, has the structure of a convex cone, we may consider generating functions of the form
	\begin{equation}\label{eq:known_gen_fcts_1}	
		f = h_L + c f_{\tilde{\MSO}_j},
	\end{equation}
	where $L$ is a convex body of revolution, $i+j=n+1$, and $c\geq 0$. 
	Another family of operations on $\MVal_i$ is the composition with \emph{Minkowski endomorphisms}.	
	These are a continuous, translation invariant, rotationally equivariant Minkowski additive operators on $\K(\R^n)$. Note that the space of Minkowski endomorphisms is precisely $\MVal_1$. Their action on the support function can be described in terms of the \emph{spherical convolution} (see \cref{sec:convol}). Kiderlen~\cite{Kiderlen2006} and Dorrek~\cite{Dorrek2017} showed that for every $\Psi\in\MVal_1$, there is a unique centered, zonal signed measure $\mu_\Psi\in\mathcal{M}(\S^{n-1})$ such that for all $K\in\K(\R^n)$,
	\begin{equation}\label{eq:MEnd_support_fct}
		h_{\Psi K}
		= h_K \ast \mu_\Psi.
	\end{equation}
	We call $\mu_\Psi$ the \emph{generating measure} of $\Psi$. Kiderlen~\cite{Kiderlen2006} showed that whenever a zonal measure $\mu\in\mathcal{M}(\S^{n-1})$ is \emph{weakly positive}, that is, positive up to the addition of a linear function, then it generates some Minkowski endomorphism.
	
	Observe that whenever $\Phi_i\in\MVal_i$ and $\Psi$ is a Minkowski endomorphism, then $\Psi\circ\Phi_i\in \MVal_i$ and its generating function is given by
	\begin{equation}\label{eq:MEnd_gen_fct}
		f_{\Psi\circ\Phi_i}
		= f_{\Phi_i} \ast\mu_{\Psi}.
	\end{equation}
	That is, the space of generating functions of Minkowski valuations is closed with respect to convolution with generating measures of Minkowski endomorphisms, including all weakly positive zonal measures. Consequently, we may consider generating functions of the form
	\begin{equation}\label{eq:known_gen_fcts_2}
		f = h_L +  f_{\tilde{\MSO}_j}\ast \mu_{\Psi},
	\end{equation}
	where $L$ is a convex body of revolution, $i+j=n+1$, and $\Psi$ is a Minkowski endomorphism. In fact, this encompasses all instances of generating functions that are known to date.
	
	The following example demonstrates that generating functions of the form \cref{eq:known_gen_fcts_2} create a much greater class than those of the form \cref{eq:known_gen_fcts_1}. This indicates the importance of taking compositions with Minkowski endomorphisms into account.
	
	\begin{exl}
		Let $n\geq 3$, $1\leq i< n-1$, and $i+j=n+1$. Due to \cref{g(-1)<0} and the fact that $\lim_{t\to 1}g_j(t)=-\infty$, there exists some $\delta>0$ such that $g_j(t)<0$ whenever $\abs{t}>1-\delta$. 
		Choose some non-trivial, zonal, positive, even measure $\mu\in\mathcal{M}(\S^{n-1})$ with a smooth density that is supported on $\{u\in\S^{n-1}:\abs{\langle e_n,u\rangle}>1-\delta \}$.
		
		Then $\mu$ is the generating measure of a Minkowski endomorphism, and thus, $f=g_j(\langle e_n,{}\cdot{}\rangle)\ast \mu$ is a generating function of the form \cref{eq:known_gen_fcts_2}. Moreover, $f$ is smooth and $f(e_n)=f(-e_n)<0$, as follows from the definition of the convolution. Suppose now that $f$ is of the form \cref{eq:known_gen_fcts_1}. Since $f$ is smooth, we must have $c=0$, but since $f$ is not weakly positive, we must also have $L=\{o\}$; consequently $f=0$, which is a contradiction.
	\end{exl}

	\section{Action on the generating function}
	\label{sec:LL_gen_fct}

	In this section, we investigate the action of the Lefschetz integration operator~$\LL$ on generating functions of Minkowski valuations, proving \cref{rho_description:intro} and its consequences. As was already indicated in the introduction, we consider a broader framework, which we introduce below.

	\subsection{Spherical valuations}
	\label{sec:Val_sph}
		
	A valuation $\varphi\in\Val$ is called \emph{smooth} if the map $\mathrm{GL}(n) \to \Val: g \mapsto g\cdot \varphi$ is infinitely differentiable, where $g\cdot \varphi$ denotes the natural action of $\mathrm{GL}(n)$ on $\Val$ (see \cref{sec:Lambda_KS}). Then, for $0\leq i\leq n$, the subspace $\Val^{\infty}_i\subseteq\Val_i$ of smooth valuations is dense in $\Val_i$ and the map
	\begin{equation*}
		\Val^{\infty}_i\to C^\infty(\mathrm{GL}(n),\Val_i): \varphi\mapsto (g\mapsto g\cdot\varphi)
	\end{equation*}
	leads to an identification of $\Val^\infty_i$ with a closed subspace of $C^\infty(\mathrm{GL}(n),\Val_i)$, endowed with the standard Fréchet topology (cf.\ \cite[Section~4.4]{Warner1972}). Hence, $\Val^\infty_i$ also becomes a Fréchet space by endowing it with the induced topology, called the G\aa{}rding topology. The space $\Val^{\infty,\mathrm{sph}}_i$ of smooth \emph{spherical} valuations is the closure of the direct sum of $\SO(n)$-irreducible subspaces of $\Val^\infty_i$ that contain a non-trivial $\SO(n-1)$-invariant element.
	
	This representation theoretic and somewhat implicit definition does not provide a clear idea of what a spherical valuation looks like. However, Schuster and Wannerer~\cite{Schuster2018} showed that these are precisely the valuations in $\Val^\infty_i$ that admit an integral representation on $\S^{n-1}$ with respect to the $i$-th area measure. In the following, we denote by $C^{\infty}_o(\S^{n-1})$ and $C^{-\infty}_o(\S^{n-1})$ the spaces of centered smooth functions and distributions, respectively.
	
	\begin{thm}[{\cite{Schuster2018}}]\label{gen_fct}
		Let $1\leq i\leq n-1$. For every valuation $\varphi\in \Val_i^{\infty,{\mathrm{sph}}}$, there exists a unique function $f_\varphi\in C^\infty_o(\S^{n-1})$ such that for all $K\in\K(\R^n)$,
		\begin{equation}\label{eq:gen_fct}
			\varphi(K)
			= \int_{\S^{n-1}} f_{\varphi}(u)~S_i(K,du).
		\end{equation}
		The map $\Val_i^{\infty,{\mathrm{sph}}}\to C^\infty_o(\S^{n-1}): \varphi\mapsto f_{\varphi}$ is an $\SO(n)$-equivariant isomorphism of topological vector spaces.
	\end{thm}
	
	We call the function $f_{\varphi}$ the \emph{generating function} of $\varphi$. As was already pointed out in the introduction, the Lefschetz derivation operator acts on the generating function as a multiple of the identity. That is, $f_{\Lambda\varphi}= i f_{\varphi}$ for $\varphi\in \Val_i^{\infty,{\mathrm{sph}}}$,
	which is a simple consequence of the Steiner formula for area measures. For this reason, we set the operator~$\Lambda$ aside and turn to the Lefschetz integration operator~$\LL$.
	
	As a direct consequence of \cref{gen_fct}, there exists a unique $\SO(n)$-equivariant endomorphism $\mathrm{T}_i$ on the topological vector space $C^\infty_o(\S^{n-1})$ such that 
	$f_{\LL\varphi}=\mathrm{T}_if_{\varphi}$
	for $\varphi\in \Val_i^{\infty,{\mathrm{sph}}}$. By setting 
	$\rho_i=\mathrm{T}_i'(\Id-\pi_1)\delta_{e_n}$,
	where $\mathrm{T}_i'$ is the continuous transpose map of $\mathrm{T}_i$ on $C^{-\infty}_o(\S^{n-1})$, one obtains the following.
	
	\begin{cor}
		For $1\leq i< n-1$, there exists a unique centered, zonal distribution $\rho_i\in C^{-\infty}(\S^{n-1})$ such that for all $\varphi\in\Val_i^{\infty,{\mathrm{sph}}}$,
		\begin{equation*}
			f_{\LL\varphi}
			= f_{\varphi} \ast \rho_i.
		\end{equation*}
	\end{cor}
	
	Interestingly enough, all the information that is needed to determine and subsequently describe the distribution $\rho_i$ can be obtained from the action of $\LL$ on the mean section operators. For $1\leq i< n-1$.
	\begin{equation*}
		\LL{\tilde{\MSO}}_j
		= \frac{\flag{j}{j-1}}{\flag{n}{n-1}} \tilde{\MSO}_{j-1},
	\end{equation*}
	where $i+j=n+1$, as was shown in \cite{Schuster2015}. By approximating the associated real valued valuations of the centered mean section operators with smooth valuations, it follows that
	\begin{equation}\label{eq:LL_MSO_gen_fct}
		f_{\tilde{\MSO}_j} \ast \rho_i
		= \frac{\flag{j}{j-1}}{\flag{n}{n-1}} f_{\tilde{\MSO}_{j-1}}.
	\end{equation}
	
	In the following, we use methods from harmonic analysis on $\S^{n-1}$ (see \cref{sec:harmonics}) to determine the spherical harmonic expansion of $\rho_i$. To that end, recall that by \cref{eq:MSO_Berg_fct}, the centered mean section operators are generated by a constant multiple of the centered Berg functions. Their multipliers were explicitly computed in \cite{Bernig2018} and \cite{Berg2018} independently: for $2\leq j\leq n$ and $k\neq 1$,
	\begin{equation*}
		a^n_k[g_j]
		= - \frac{\pi^{\frac{n-j}{2}}(j-1)}{4} \frac{\Gamma\big(\frac{n-j+2}{2}\big)\Gamma\big(\frac{k-1}{2}\big)\Gamma\big(\frac{k+j-1}{2}\big)}{\Gamma\big(\frac{k+n-j+1}{2}\big)\Gamma\big(\frac{k+n+1}{2}\big)}.
	\end{equation*}
	These numbers are all non-zero, and thus, by a simple division, for all $k\neq 1$,
	\begin{equation*}
		a^n_k[\rho_i]
		= \frac{\flag{j}{j-1}}{\flag{n}{n-1}} \frac{a^n_k[f_{\tilde{\MSO}_{j-1}}]}{a^n_k[f_{\tilde{\MSO}_j}]}
		= c_{n,i} \frac{\Gamma\big(\frac{k+i}{2}\big)\Gamma\big(\frac{k+n-i-1}{2}\big)}{\Gamma\big(\frac{k+i+1}{2}\big)\Gamma\big(\frac{k+n-i}{2}\big)},
	\end{equation*}
	where $i+j=n+1$ and the constant $c_{n,i}>0$ only depends on $n$ and $i$. The above was already done by Berg, Parapatits, Schuster, and Weberndorfer~\cite{Berg2018}. Our contribution will be to extract the information of \cref{rho_description:intro} from the spherical harmonic expansion of $\rho_i$.

	\subsection{A Legendre type differential equation}
	\label{sec:rho_ODE}

	Our aim is now to establish the differential equation~\cref{eq:rho_ODE:intro}. To that end, it is convenient to renormalize $\rho_i$ and manipulate its linear part: We define an auxiliary distribution $\tilde{\rho}_i\in C^{-\infty}(\S^{n-1})$ in terms of its multipliers by
	\begin{equation}\label{eq:rho_n.i}
		a^n_k[\tilde{\rho}_i]
		= \frac{\Gamma\big(\frac{k+i}{2}\big)\Gamma\big(\frac{k+n-i-1}{2}\big)}{\Gamma\big(\frac{k+i+1}{2}\big)\Gamma\big(\frac{k+n-i}{2}\big)},
		\qquad k\geq 0.
	\end{equation}
	
	Next, note that we can lift every test function $\psi\in\D(-1,1)$ to a smooth function $\psi(\langle e_n,{}\cdot{}\rangle)\in C^\infty(\S^{n-1})$. In the following, we let
	\begin{equation*}
		I_{\ast}
		= \{(n,i)\in\N\times\N : n\geq 3,\ 1\leq i< n-1\}
	\end{equation*}
	denote the set of admissible index pairs. Then we define a family of distributions ${\rho}_{n,i}\in\D'(-1,1)$, where $(n,i)\in I_\ast$, by
	\begin{equation*}
		\langle \psi , \rho_{n,i} \rangle
		= \langle \psi(\langle e_n,{}\cdot{}\rangle) , \tilde{\rho}_i \rangle,
		\qquad \psi\in\D(-1,1).
	\end{equation*}
	This construction on the interval $(-1,1)$ will allow us to relate distributions $\tilde{\rho}_i$ that live on spheres $\S^{n-1}$ of different dimensions.
	
	We continue with two simple, yet crucial observations about the family $\rho_{n,i}$ that are immediate from their definition and the fact that every distribution is uniquely determined by its sequence of multipliers. Firstly, for $(n,i)\in I_\ast$,
	\begin{equation}\label{eq:rho_index_reflection}
		\rho_{n,i}
		= \rho_{n,n-i-1}.
	\end{equation}
	Secondly, we compare the multipliers of ${\rho}_{n,n-2}$ with those of $g_{n-1}$ and the \emph{box operator} $\square_n=\frac{1}{n-1}\Delta_{\S^{n-1}}+\Id$ (see~\cref{sec:harmonics}). Its action on zonal functions (see, e.g., \cite[Lemma~5.3]{OrtegaMoreno2021}) can be described in terms of the differential operator
	\begin{equation*}
		\overline{\square}_n
		= \frac{1}{n-1}(1-t^2)\frac{d^2}{dt^2} - t\frac{d}{dt} + \Id.
	\end{equation*}
	Namely, for all test functions $\psi\in\D(-1,1)$, we have that
	\begin{equation*}
		[\square_n\psi(\langle e_n,{}\cdot{}\rangle)](u)
		= (\overline{\square}_n\psi)(\langle e_n,u\rangle),
		\qquad u\in\S^{n-1}.
	\end{equation*}
	In this way, for all $n\geq 3$, we obtain that the distribution $\rho_{n,n-2}$ is a $C^\infty(-1,1)$ function and that for all $t\in (-1,1)$,
	\begin{equation}\label{eq:rho_n.n-2}
		\rho_{n,n-2}(t)
		= \frac{n-1}{2(n-2)} \overline{\square}_n g_{n-1}(t) + \frac{n}{2\omega_{n-1}}t.
	\end{equation}
	
	Next, we establish the smoothness of the distributions $\rho_{n,i}$ on $(-1,1)$, their behavior at the end points, and a remarkably simple recurrence relation.
	
	\begin{lem} \label{rho_smooth_recurrence}
		Let $(n,i)\in I_\ast$.
		\begin{enumerate}[label=\upshape(\roman*)]
			\item \label{rho_smooth}
			$\rho_{n,i}\in C^\infty(-1,1)$ and for all $m\geq 0$,
			\begin{align}
				\begin{split}\label{eq:rho_t->+-1}
					\lim_{t\to -1} (1-t^2)^{\frac{n-2}{2} + m} \rho_{n,i}^{(m)}(t)
					&= 0, \\
					\lim_{t \to +1} (1-t^2)^{\frac{n-2}{2} + m} \rho_{n,i}^{(m)}(t)
					&= 2^{m-2} \frac{\Gamma\big(\frac{n-2}{2} + m\big)}{\pi^{\frac{n-2}{2}}}.
				\end{split}
			\end{align}
			\item \label{rho_recurrence}
			For all $t\in (-1,1)$,
			\begin{equation} \label{eq:rho_recurrence}
				\rho_{n+2,i+1}(t)
				= \frac{1}{2\pi} \rho_{n,i}'(t).
			\end{equation}			
		\end{enumerate}
	\end{lem}
	
	In order to show the recurrence relation \cref{eq:rho_recurrence}, we need the following multiplier relation, which we prove in~\cref{sec:harmonics}.
	
	\begin{lem}\label{multiplier_diff}
		If $n\geq 3$ and $g\in C^1(-1,1)$ is such that $(1-t^2)^{\frac{n-1}{2}} g'(t)$ is integrable on $(-1,1)$, then $(1-t^2)^{\frac{n-3}{2}}g(t)$ is integrable on $(-1,1)$ and for all $k\geq 0$,
		\begin{equation}\label{eq:multiplier_diff}
			a^{n+2}_k[g']
			= 2\pi a^n_{k+1}[g].
		\end{equation}
	\end{lem}
	
	We want to point out that \cref{eq:rho_index_reflection}, \cref{eq:rho_n.n-2}, and \cref{eq:rho_recurrence} together result in a complete recursion such that every $\rho_{n,i}$ can be traced back to the Berg functions. In this way, we will obtain many of the desired properties inductively. The following lemma provides the necessary induction scheme on $I_\ast$; the statement is obvious, so we omit the proof.
	
	\pagebreak
	
	\begin{lem}\label{rho_induction_scheme}
		Let $I\subseteq I_{\ast}$ and suppose that for all $(n,i)\in I_\ast$, the following holds:
		\begin{itemize}
			\item $(n,n-2)\in I$,
			\item $(n,i)\in I$ if and only if $(n,n-i-1)\in I$,
			\item whenever $(n,i)\in I$, then also $(n+2,i+1)\in I$.
		\end{itemize}
		Then $I=I_{\ast}$.
	\end{lem}

	\begin{proof}[Proof of \cref{rho_smooth_recurrence}]
		First, we show that for some fixed $(n,i)\in I_\ast$, the assertion of \cref{rho_smooth} implies the assertion of \cref{rho_recurrence}. For this purpose, suppose that $\rho_{n,i}\in C^\infty(-1,1)$ and satisfies \cref{eq:rho_t->+-1}. Then the requirements of \cref{multiplier_diff} are met, and thus, \cref{eq:multiplier_diff} shows that for all $k\geq 0$,
		\begin{equation*}\label{eq:rho_smooth_recurrence:proof}
			\frac{1}{2\pi}a^{n+2}_k[\rho_{n,i}']
			= a^n_{k+1}[\rho_{n,i}]
			= \frac{\Gamma\big(\frac{k+i+1}{2}\big)\Gamma\big(\frac{k+n-i}{2}\big)}{\Gamma\big(\frac{k+i+2}{2}\big)\Gamma\big(\frac{k+n-i+1}{2}\big)}
			= a^{n+2}_k[\rho_{n+2,i+1}].
		\end{equation*}
		This proves \cref{eq:rho_recurrence} in the sense of distributions. Since we assumed $\rho_{n,i}$ to be smooth, it follows that $\rho_{n+2,i+1}$ is also smooth, and thus, identity \cref{eq:rho_recurrence} holds pointwise.
		
		We will now prove \cref{rho_smooth_recurrence} through the induction scheme from above. To that end, denote by $I \subseteq I_\ast$ the set of admissible index pairs for which \cref{rho_smooth_recurrence} (or equivalently, \cref{rho_smooth}) is true. Due to \cref{eq:rho_index_reflection}, we have that $(n,i)\in I$ if and only if $(n,n-i-1)\in I$. From \cref{eq:rho_n.n-2} and by applying the general Leibniz rule, it follows that $\rho_{n,n-2}\in C^\infty(-1,1)$ and for all $m\geq0$,
		\begin{multline*}
			\rho_{n,n-2}^{(m)}(t)
			= \frac{1}{2(n-2)}(1-t^2)g_{n-1}^{(m+2)}(t) - \frac{2m+n-1}{2(n-2)}tg_{n-1}^{(m+1)}(t)\\ - \frac{(m+n-1)(m-1)}{2(n-2)}g_{n-1}^{(m)}(t) + \frac{n}{2\omega_{n-1}}\frac{d^mt}{dt^m}.
		\end{multline*}
		Now we multiply both sides with $(1-t^2)^{\frac{n-2}{2}+m}$, pass to the limit $t\to \pm 1$, and employ \cref{eq:g_t->+-1}. This verifies \cref{eq:rho_t->+-1} in the instance where $i=n-2$ for all $m\geq 0$, and thus, $(n,n-2)\in I$. 
		
		Suppose now that $(n,i)\in I$. Then, as we have established at the beginning of the proof, $\rho_{n+2,i+1}\in C^\infty(-1,1)$ and for all $m\geq 0$,
		\begin{align*}
			&\lim_{t\to+1} (1-t^2)^{\frac{n}{2}+m}\rho_{n+2,i+1}^{(m)}(t)
			= \frac{1}{2\pi}\lim_{t\to +1} (1-t^2)^{\frac{n-2}{2}+(m+1)}\rho_{n,i}^{(m+1)}(t) \\
			&\qquad = \frac{1}{2\pi} 2^{(m+1)-2} \frac{\Gamma\big(\frac{n-2}{2} + (m+1)\big)}{\pi^{\frac{n-2}{2}}}
			= 2^{m-2}\frac{\Gamma\big(\frac{n}{2}+m\big)}{\pi^{\frac{n}{2}}},
		\end{align*}
		where we applied the appropriate instance of recurrence relation~\cref{eq:rho_recurrence}. Similarly,
		\begin{equation*}
			\lim_{t\to-1} (1-t^2)^{\frac{n}{2}+m}\rho_{n+2,i+1}^{(m)}(t)
			= \frac{1}{2\pi}\lim_{t\to -1} (1-t^2)^{\frac{n-2}{2}+(m+1)}\rho_{n,i}^{(m+1)}(t)
			= 0,
		\end{equation*}
		and thus, we obtain that also $(n+2,i+1)\in I$. \cref{rho_induction_scheme} now yields $I=I_\ast$, which concludes the argument.
	\end{proof}
	
	We now establish the Legendre type differential equation~\cref{eq:rho_ODE:intro}.
	
	\begin{thm}\label{rho_ODE}
		Let $(n,i)\in I_\ast$. Then for all $t\in (-1,1)$,
		\begin{equation}\label{eq:rho_ODE}
			(1-t^2)\rho_{n,i}''(t) - n t \rho_{n,i}'(t) - i(n-i-1)\rho_{n,i}(t)
			= 0.
		\end{equation}
	\end{thm}
	\begin{proof}
		First, define a second order differential operator
		\begin{equation*}
			D_{n,i}
			= (1-t^2)\frac{d^2}{dt^2} - nt\frac{d}{dt} - i(n-i-1)\Id
		\end{equation*}
		and denote by $I$ the set of all index pairs $(n,i)\in I_\ast$ for which $D_{n,i}\rho_{n,i}=0$.		
		Due to \cref{eq:rho_index_reflection} and the fact that $D_{n,i}=D_{n,n-i-1}$, we have that $(n,i)\in I$ if and only if $(n,n-i-1)\in I$. Note that
		\begin{equation*}
			(n-1)D_{n,n-2}\overline{\square}_n
			= (n-2)D_{n+1,n-1}\overline{\square}_{n-1},
		\end{equation*}
		as can be shown by direct computation. Therefore, by applying \cref{eq:rho_n.n-2}, we obtain
		\begin{align*}
			& D_{n,n-2}\rho_{n,n-2}(t)
			= \frac{n-1}{2(n-2)} D_{n,n-2}\overline{\square}_n g_{n-1}(t) + \frac{n}{2\omega_{n-1}} D_{n,n-2}t \\
			&\qquad = \frac{1}{2} D_{n+1,n-1} \overline{\square}_{n-1}g_{n-1}(t) + \frac{n}{2\omega_{n-1}} D_{n,n-2}t \\
			&\qquad = - \frac{n-1}{2\omega_{n-1}} D_{n+1,n-1}t + \frac{n}{2\omega_{n-1}} D_{n,n-2}t
			= 0,
		\end{align*}
		where the third equality is due to \cref{eq:Berg_fct_ODE}. It follows that $(n,n-2)\in I$.
		
		Suppose now that $(n,i)\in I$. Note that $D_{n+2,i+1} \circ \frac{d}{dt}
		= \frac{d}{dt} \circ D_{n,i}$, as can again be shown by direct computation. By employing recurrence relation \cref{eq:rho_recurrence}, we get that for all $t\in(-1,1)$,
		\begin{equation*}
			D_{n+2,i+1}\rho_{n+2,i+1}(t)
			= \frac{1}{2\pi} D_{n+2,i+1}\rho_{n,i}'(t)
			= \frac{1}{2\pi} (D_{n,i}\rho_{n,i})'(t)
			= 0,
		\end{equation*}
		and thus, $(n+2,i+1)\in I$. By \cref{rho_induction_scheme}, we have that $I=I_\ast$.
	\end{proof}
	
	The differential equation \cref{eq:rho_ODE}, combined with the behavior at the end points has the following notable interpretation: If we lift $\rho_{n,i}$ to a zonal function on the unit sphere in $\R^{n+1}$, then it is the Green's function of some strictly elliptic Helmholtz type operator.
	
	\begin{cor}\label{rho_Green}
		Let $(n,i)\in I_\ast$. Then $\rho_{n,i}(\langle e_{n+1},{}\cdot{}\rangle) \in L^1(\S^n)$ and
		\begin{equation}\label{eq:rho_Green}
			(-\Delta_{\S^{n}} + i(n-i-1)\Id )\rho_{n,i}(\langle e_{n+1},{}\cdot{}\rangle)
			= \pi\delta_{e_{n+1}}
		\end{equation}
		in the sense of distributions on $\S^n$.
	\end{cor}
	\begin{proof}
		Due to \cref{rho_smooth_recurrence}~\ref{rho_smooth}, the function $(1-t^2)^{\frac{n-2}{2}}\rho_{n,i}(t)$ is smooth and bounded, and thus, integrable on $(-1,1)$. By a change to spherical cylinder coordinates, this implies that $\rho_{n,i}(\langle e_{n+1},{}\cdot{}\rangle)$ is integrable on $\S^n$.
		
		Next, observe that both sides of \cref{eq:rho_Green} are zonal distributions on $\S^n$. Hence, it suffices to test the identity on zonal smooth functions $\phi\in C^\infty(\S^n)$. Then $\phi$ is of the form $\phi(u)=\psi(\langle e_{n+1},u\rangle)$ for some $\psi\in C^\infty(-1,1)$. By a change to spherical cylinder coordinates,
		\begin{align*}
			&\left< \phi, (\Delta_{\S^{n}} - i(n-i-1)\Id )\rho_{n,i}(\langle e_{n+1},{}\cdot{}\rangle) \right>_{\S^n} \\
			&\qquad = \left< (\Delta_{\S^{n}} - i(n-i-1)\Id )\phi, \rho_{n,i}(\langle e_{n+1},{}\cdot{}\rangle) \right>_{\S^n} \\
			&\qquad = \omega_{n}\int_{(-1,1)} \left(((1-t^2)^{\frac{n}{2}}\psi'(t))' - i(n-i-1)\psi(t)(1-t^2)^{\frac{n-2}{2}} \right) \rho_{n,i}(t) dt.
		\end{align*}
		We consider the first part of the integral and perform integration by parts. Then
		\begin{align*}
			&\int_{(-1,1)} ((1-t^2)^{\frac{n}{2}}\psi'(t))' \rho_{n,i}(t)dt \\
			&\qquad = \int_{(-1,1)} \psi'(t)\rho_{n,i}'(t)(1-t^2)^{\frac{n}{2}}dt  + \left[ (1-t^2)^{\frac{n}{2}}\psi'(t)\rho_{n,i}(t) \right]_{t=-1}^{t=1},
		\end{align*}
		where the marginal terms are to be understood as limits. If $\gamma$ is a geodesic in $\S^n$ such that $\gamma(0)=-e_{n+1}$, then
		\begin{align*}
			\frac{d}{d\theta} (\phi\circ\gamma)(\theta)
			= \frac{d}{d\theta}\psi(-\cos\theta)
			= (\sin\theta) \psi'(-\cos\theta)
			= (1-t^2)^{\frac{1}{2}}\psi'(t),
		\end{align*}
		where we applied the change of variables $t=-\cos\theta$ in the final equality.  Since $\phi\in C^\infty(\S^n)$, the final expression is bounded, so by \cref{eq:rho_t->+-1}, the marginal terms vanish. Integrating by parts a second time yields 
		\begin{align*}
			&\int_{(-1,1)} ((1-t^2)^{\frac{n}{2}}\psi'(t))' \rho_{n,i}(t)dt \\
			&\qquad = \int_{(-1,1)} \psi(t)((1-t^2)^{\frac{n}{2}}\rho_{n,i}'(t))' dt - \left[(1-t^2)^{\frac{n}{2}}\psi(t)\rho_{n,i}'(t)\right]_{t=-1}^{t=1} \\
			&\qquad = \int_{(-1,1)} \psi(t)((1-t^2)^{\frac{n}{2}}\rho_{n,i}'(t))' dt - \frac{\pi}{\omega_n}\psi(1),
		\end{align*}
		where the marginal terms are again to be understood as limits. The final equality is due to \cref{eq:rho_t->+-1} and the fact that $\psi$ is bounded. By the differential equation \cref{eq:rho_ODE}, the remaining integral terms add up to zero, and thus, we obtain that
		\begin{align*}
			\left< \phi, (\Delta_{\S^{n}} - i(n-i-1)\Id )\rho_{n,i}(\langle e_{n+1},{}\cdot{}\rangle) \right>_{\S^n}
			= \pi\phi(e_{n+1}),
		\end{align*}
		which completes the proof.
	\end{proof}

	\subsection{Description of \texorpdfstring{$\rho_i$}{rhoi}}
	\label{sec:rho_description}

	In the following, we use the Legendre type differential equation \cref{eq:rho_ODE} to obtain the desired description of the distributions $\rho_i$; we prove what is left to prove of \cref{rho_description:intro} and discuss its consequences.
	
	First, we want to point out that the differential equation \cref{eq:rho_ODE}, combined with the behavior at the end points $t=1$ and $t=-1$, determines each function $\rho_{n,i}$ uniquely. In particular, we obtain an interesting representation for the family $\rho_{n,i}$. For parameters $\lambda, \mu\in\R$, the \emph{associated Legendre function} $\tilde{P}^\mu_\lambda$ is defined as
	\begin{equation*}
		\tilde{P}^{\mu}_{\lambda}(t)
		= \frac{e^{\mathrm{i}\pi\frac{\mu}{2}}}{\Gamma(1-\mu)} \left(\frac{1+t}{1-t}\right)^{\!\frac{\mu}{2}}{}_{2}F_{1}\!\left(-\lambda,\lambda+1,1-\mu,\frac{1-t}{2}\right),
	\end{equation*}
	where ${}_{2}F_{1}$ is the hypergeometric function (cf.~\cite[8.1.2]{Abramowitz1964}). We show that $\rho_{n,i}$ is a constant multiple of a reflection of an associated Legendre function.
	
	\begin{prop}\label{rho_representation}
		Let $(n,i)\in I_{\ast}$. Then for all $t\in(-1,1)$,
		\begin{equation}\label{eq:rho_representation}
			\rho_{n,i}(t)
			= \frac{\Gamma(i)\Gamma(n-i-1)}{4(2\pi)^{\frac{n-2}{2}}} e^{\mathrm{i}\pi\frac{n-2}{4}} (1-t^2)^{-\frac{n-2}{4}} \tilde{P}^{1-\frac{n}{2}}_{i-\frac{n}{2}}(-t).
		\end{equation}
	\end{prop}
	\begin{proof}
		For fixed $(n,i)\in I_\ast$, we consider the function $y(t)=(1-t^2)^{\frac{n-2}{4}}\rho_{n,i}(t)$. \linebreak
		A simple transformation of \cref{eq:rho_ODE} shows that $y$ is a solution to Legendre's differential equation: For all $t\in(-1,1)$,
		\begin{equation}\label{eq:Legendre_ODE}
			(1-t^2)y''(t) - 2ty'(t) + \bigg(\lambda(\lambda+1) - \frac{\mu^2}{1-t^2} \bigg)y(t)
			= 0
		\end{equation}
		with parameters $\lambda=i-\frac{n}{2}$ and $\mu=1-\frac{n}{2}$. 
		
		Moreover, the associated Legendre function $\tilde{P}^{\mu}_{\lambda}(t)$, and thus, also its reflection $\tilde{P}^{\mu}_{\lambda}(-t)$
		solve this differential equation (cf.~\cite[8.1.1]{Abramowitz1964}). In order to describe their behavior at the end points, we recall Gauss' summation theorem (cf.~\cite[15.1.20]{Abramowitz1964}), which states that
		\begin{equation*}
			\lim_{t\to+1} {}_2F_{1}(a,b,c,t)
			= \frac{\Gamma(c)\Gamma(c-a-b)}{\Gamma(c-a)\Gamma(c-b)}
		\end{equation*}
		whenever $c-a-b>0$. As an instance of this identity and due to the simple fact that ${}_2F_{1}(a,b,c,0)=1$, we obtain that		 
		\begin{align}\label{eq:LegendreP_t->+-1}
			\begin{split}
				\lim_{t\to -1} (1-t^2)^{\frac{n-2}{4}}\tilde{P}^{\mu}_{\lambda}(t)
				&= e^{-\mathrm{i}\pi\frac{n-2}{4}} \frac{2^{\frac{n-2}{2}}\Gamma\big(\frac{n-2}{2}\big)}{\Gamma(i)\Gamma(n-i-1)},  \\
				\lim_{t\to +1} (1-t^2)^{\frac{n-2}{4}}\tilde{P}^{\mu}_{\lambda}(t)
				&= 0.
			\end{split}			
		\end{align}
		This entails that $\tilde{P}^{\mu}_{\lambda}(t)$ and $\tilde{P}^{\mu}_{\lambda}(-t)$ are linearly independent, and thus, they form a basis of the two-dimensional space of solutions to \cref{eq:Legendre_ODE}. Consequently, the function $y(t)$ is a linear combination of those two functions. By comparing \cref{eq:rho_t->+-1} and \cref{eq:LegendreP_t->+-1}, we observe that $y(t)$ must in fact be a constant multiple of $\tilde{P}^{\mu}_{\lambda}(-t)$, and we obtain the respective multiplicative constant from a simple division.		
	\end{proof}	
	
	In \cref{eq:rho_t->+-1}, we have shown that the functions $\rho_{n,i}$ have a singularity at $t=1$ and determined the precise asymptotic behavior. In order to investigate their behavior at $t=-1$, we perform the substitution $t=-\cos\theta$ (just like we did for the Berg functions in \cref{sec:Berg_fcts}). That is, we define a family of functions $\hat{\rho}_{n,i}$, where $(n,i)\in I_\ast$, by
	\begin{equation*}
		\hat{\rho}_{n,i}(\theta)
		= \rho_{n,i}(-\cos\theta),
		\qquad\theta\in (-\pi,\pi)\backslash\{0\}.
	\end{equation*}
	Applying the transformation $t=-\cos\theta$ to the recursion for the family $\rho_{n,i}$ yields the following recursion for the family $\hat{\rho}_{n,i}$:
	\begin{equation}\label{eq:rho_index_reflection_arc}
		\hat{\rho}_{n,i}
		= \hat{\rho}_{n,n-i-1},
	\end{equation}
	\begin{equation}\label{eq:rho_n.n-2_arc}
		\hat{\rho}_{n,n-2}(\theta)
		= \frac{1}{2(n-2)}\hat{g}_{n-1}''(\theta) + \frac{1}{2} \frac{\hat{g}_{n-1}'(\theta)}{\tan\theta} + \frac{n-1}{2(n-2)}\hat{g}_{n-1}(\theta)  - \frac{n}{2\omega_{n-1}}\cos\theta,
	\end{equation}
	\begin{equation}\label{eq:rho_recurrence_arc}
		\hat{\rho}_{n+2,i+1}(\theta)
		= \frac{1}{2\pi} \frac{\hat{\rho}_{n,i}'(\theta)}{\sin\theta}.
	\end{equation}
	
	This recursion shows that the functions $\hat{\rho}_{n,i}$ extend to even analytic functions on $(-\pi,\pi)$.
	To that end, denote by $I$ the set of all index pairs $(n,i)\in I_\ast$ for which this is the case. Due to \cref{eq:rho_index_reflection_arc}, we have that $(n,i)\in I$ if and only if $(n,n-i-1)\in I$. We have shown in \cref{sec:Berg_fcts} that $\hat{g}_{n-1}$ is an even analytic function on $(-\pi,\pi)$. Consequently, its derivative $\hat{g}_{n-1}'$ has a simple zero in $\theta=0$ which cancels with the zero of the tangent function, thus $(n,n-2)\in I$. Applying a similar argument to the recurrence relation \cref{eq:rho_recurrence_arc} shows that if $(n,i)\in I$, then also $(n+2,i+1)\in I$. \cref{rho_induction_scheme} now yields $I=I_\ast$.

	We can now complete the proof of \cref{rho_description:intro}.
	
	\begin{thm}\label{rho_description}
		Let $1\leq i<n-1$. Then $\rho_i$ is an $L^1(\S^{n-1})$ function which is smooth on $\S^{n-1}\setminus \{e_n\}$ and strictly positive up to the addition of a linear function.
	\end{thm}
	\begin{proof}
		Due to \cref{rho_smooth_recurrence}~\ref{rho_smooth}, the function $\rho_{n,i}(\langle e_n,{}\cdot{}\rangle)$ is integrable on $\S^{n-1}$ and smooth on $\S^{n-1}\setminus\{ \pm e_n\}$. In particular, it defines a zonal distribution on $\S^{n-1}$ which must coincide with $\tilde{\rho}_i$, since zonal distributions are uniquely determined by their multipliers.
		
		Observe that $\hat{\rho}_{n,i}(\theta)$, as an even analytic function on $(-\pi,\pi)$, can be written as a power series of $\theta^2$. Thus, there exists a unique analytic smooth function $q_{n,i}$ such that ${\rho}_{n,i}(-\cos\theta)=q_{n,i}(\theta^2)$. Consequently,
		\begin{equation*}
			\tilde{\rho}_i(u)
			= \rho_{n,i}(\langle e_n,u\rangle)
			= \rho_{n,i}( - \cos d_{\S^{n-1}}(-e_n,u))
			= q_{n,i}(d_{\S^{n-1}}(-e_n,u)^2)
		\end{equation*}
		for all $u\in\S^{n-1}\setminus \{e_n\}$.
		Together with \cref{dist_square_smooth}, this shows that $\tilde{\rho}_i$ is smooth around the south pole.
		
		We have shown that $\rho_{n,i}$ satisfies the second-order elliptic differential equation \cref{eq:rho_ODE} on $(-1,1)$. Moreover, due to the representation formula~\cref{eq:rho_representation},		
		\begin{equation*}
			\lim_{t\to -1}\rho_{n,i}(t)
			= \frac{\Gamma(i)\Gamma(n-i-1)}{2^n \pi^{\frac{n-2}{2}}\Gamma\big(\frac{n}{2}\big)}.
		\end{equation*}
		Combined with \cref{eq:rho_t->+-1}, we see that
		\begin{equation*}
			\lim_{t\to -1} \rho_{n,i}(t) > 0
			\qquad\text{and}\qquad
			\lim_{t \to +1} \rho_{n,i}(t) > 0,
		\end{equation*}
		so the strong maximum principle (cf.~\cite[Section~6.4.2]{Evans2010}) implies that the function $\rho_{n,i}$ attains a strictly positive minimum on $(-1,1)$. Since $\tilde{\rho}_i$ is just a renormalization of $\rho_i$ up to the addition of a linear function, the claim follows.
	\end{proof}
	
	As was pointed out in the introduction, this result has some interesting consequences for Minkowski valuations that we will now touch upon, revisiting our discussion in \cref{sec:MVali}. Since $\rho_i$ is weakly positive, it is the density of the generating measure of some Minkowski endomorphism $\Psi^{(i)}$. Hence, if a Minkowski valuation $\Phi_i\in\MVal_i$ is generated by some body of revolution $L\in\K(\R^n)$, then, by \cref{eq:LL_Lambda_gen_fct:intro} and \cref{eq:MEnd_gen_fct},
	\begin{equation*}
		f_{\LL\Phi_i}
		= f_{\Phi_i}\ast\rho_i
		= h_L\ast \mu_{\Psi^{(i)}}
		= h_{\Psi^{(i)}L},
	\end{equation*}
	and thus, $\LL\Phi$ is again generated by some body of revolution. More generally, we obtain that all known examples of generating functions are preserved under the action of the Lefschetz operators.
	
	\begin{cor}\label{known_gen_fcts_preserved}
		Generating functions of the form \cref{eq:known_gen_fcts_2} are preserved under the action of the Lefschetz operators.
	\end{cor}
	\begin{proof}
		Consider a Minkowski valuation $\Phi_i\in\MVal_i$ with a generating function of the form \cref{eq:known_gen_fcts_2}, that is, $f_{\Phi_i} = h_L +  f_{\tilde{\MSO}_j}\ast \mu_{\Psi}$, where $L$ is a convex body of revolution, $\Psi\in\MVal_1$, and $i+j=n+1$.
		Due to \cref{eq:LL_Lambda_gen_fct:intro}, \cref{eq:MEnd_support_fct}, and \cref{eq:LL_MSO_gen_fct},
		\begin{gather*}
			f_{\LL\Phi_i}
			= h_L \ast \rho_i + (f_{\tilde{\MSO}_j}\ast\rho_i)\ast \mu_{\Psi}
			= h_{\Psi^{(i)} L} +  a_{n,i} f_{\tilde{\MSO}_{j-1}}\ast\mu_{\Psi}, \\
			f_{\Lambda\Phi_i}
			= if_{\Phi_i}
			= h_{iL} +  f_{\tilde{\MSO}_{j+1}} \ast (\rho_{i-1} \ast \mu_{\Psi})
			= h_{iL} + b_{n,i} f_{\tilde{\MSO}_{j+1}} \ast \mu_{\Psi\circ\Psi^{(i-1)}},
		\end{gather*}	
		with some constants $a_{n,i}$ and $b_{n,i}$. Consequently, both $f_{\LL\Phi_i}$ are $f_{\Lambda\Phi_i}$ are again of the form \cref{eq:known_gen_fcts_2}.		
	\end{proof}
	
	Lastly, the combined Lefschetz operators $\Lambda\circ\LL$ and $\LL\circ\Lambda$ act on Minkowski valuations as a composition with  Minkowski endomorphisms.
	
	\begin{cor}
		Let $1\leq i< n-1$. There exists $\Psi^{(i)}\in\MVal_1$ such that
		\begin{equation*}
			\Lambda( \LL\Phi_i)
			= i\Psi^{(i)}\circ\Phi_i \qquad \text{and} \qquad \LL(\Lambda\Phi_{i+1}) = i\Psi^{(i)} \circ \Phi_{i+1}
		\end{equation*}
		for every $\Phi_i\in\MVal_i$ and $\Phi_{i+1}\in\MVal_{i+1}$.
	\end{cor}
	\begin{proof}
		Due to \cref{eq:LL_Lambda_gen_fct:intro} and \cref{eq:MEnd_gen_fct},
		\begin{gather*}
			f_{\Lambda(\LL\Phi_i)}
			= if_{\LL\Phi_i}
			= if_{\Phi_i}\ast\rho_i
			= f_{i\Psi^{(i)}\circ\Phi_i}, \\
			f_{\LL(\Lambda\Phi_{i+1})}
			= f_{\Lambda\Phi_i}\ast\rho_i
			= if_{\Phi_i}\ast\rho_i
			= f_{i\Psi^{(i)}\circ \Phi_i}.
		\end{gather*}
		Since a Minkowski valuation is uniquely determined by its generating function, this proves the corollary.
	\end{proof}

	\section{Klain--Schneider functions of Minkowski valuations}
	\label{sec:KS_gen_fct}

	In this final section, we discuss the connection between Klain--Schneider functions and generating functions. For a smooth spherical valuation $\varphi\in\Val_i^{\infty,\mathrm{sph}}$, where $1\leq i\leq n-1$, we can express its Klain--Schneider function in terms of its generating function by
	\begin{equation}\label{eq:KS_spherical}
		\KS_{\varphi}(E,u)
		= \frac{1}{\binom{n-1}{i}} (\pi_{E,-i} f_{\varphi})(u),
		\qquad (E,u)\in\Fl_{n,i+1},
	\end{equation}
	as can easily be deduced from \cref{eq:area_meas_lift}, and is a special case of \cref{eq:KS_mixed_vol}.  
	
	Now we turn to Minkowski valuations. For $\Phi\in\MVal$ and $v\in\S^{n-1}$, we define a valuation $\varphi^v\in\Val$ by $\varphi^v(K)=h_{\Phi K}(v)$, generalizing its associated real valued valuation.
	The Klain function of an \emph{even} Minkowski valuation $\Phi\in\MVal_i$ can then be defined as a continuous function on $\Gr_{n,i}\times \S^{n-1}$ by
	\begin{equation*}
		\Kl_{\Phi}(F,v)
		=\Kl_{\varphi^v}(F),
		\qquad F\in\Gr_{n,i},\qquad v\in\S^{n-1}.
	\end{equation*}
	For fixed $F\in\Gr_{n,i}$, the function $\Kl_{\Phi}(F,{}\cdot{})$ is the support function of a convex body that is invariant under all rotations stabilizing~$F$, called the \emph{Klain body} of $\Phi$. It was introduced by Schuster and Wannerer~\cite{Schuster2015} and is given by $1/\kappa_i \Phi(B^n\cap F)$. Every even $\Phi\in\MVal_i$ is uniquely determined by its Klain body.
	
	In order to also encompass non-even Minkowski valuations, like the mean section operators, we define the Klain--Schneider function of a Minkowski valuation $\Phi\in\MVal_i$ as
	\begin{equation*}
		\KS_{\Phi}((E,u),v)
		= \KS_{\varphi^v}(E,u),
		\qquad (E,u)\in\Fl_{n,i+1},
		\qquad v\in\S^{n-1}.
	\end{equation*}
	Then $\KS_{\Phi}$ is a continuous function on $\Fl_{n,i+1}\times \S^{n-1}$ that determines $\Phi$ uniquely. Moreover, due to the rotationally equivariance of~$\Phi$, we have that
	\begin{equation*}
		\KS_{\Phi}(\vartheta (E,u),\vartheta v)
		= \KS_{\Phi}((E,u),v)
	\end{equation*}
	for every $\vartheta\in\SO(n)$. Since the group of rotations $\SO(n)$ acts transitively on $\Fl_{n,i+1}$, this implies that whenever we fix some $(E,u)\in\Fl_{n,i+1}$, the function $\KS_{\Phi}((E,u),{}\cdot{})\in C(\S^{n-1})$ already contains all the information of $\KS_{\Phi}$. In this way, it also makes sense to consider the Klain--Schneider function of $\Phi$ as this continuous function on the unit sphere that is invariant under all rotations stabilizing $(E,u)$.
	
	We consider some concrete examples. In~\cite{Schuster2015}, the Klain bodies of the projection body maps and the even parts of the mean section operators were computed. Below, we compute the respective Klain--Schneider functions (taking into account the odd part of the mean section operators) in a very direct way.
	
	\begin{exl}
		For $1\leq i\leq n-1$, the $i$-th projection body map $\Pi_i$ is even, and thus, for every $(E,u)\in\Fl_{n,i+1}$ and $v\in\S^{n-1}$,
		\begin{align*}
			&\KS_{\Pi_i}((E,u),v)
			= \frac{1}{2} \Kl_{\Pi_i}(E\cap u^\perp,v) \\
			&\qquad 
			= \frac{1}{2\kappa_i} V_i((B^n \cap (E\cap u^\perp)^\perp) | v^\perp)
			= \frac{1}{2} \norm{v | (E^\perp \vee u)}.
		\end{align*}
	\end{exl}


	\begin{exl}
		Let $1\leq i\leq n-1$ and $i+j=n+1$. For a subspace $E\in\Gr_{n,i+1}$ and a convex body $K\in\K(E)$, an application of~\cref{eq:AGr_reparametrizaton} shows that for $v\in\S^{n-1}$,
		\begin{equation*}
			h(\tilde{\MSO}_j K,v)
			= \frac{\flag{j}{2}}{\flag{n}{i+1}} h( \tilde{\MSO}^E_2 K, v )
			= \frac{\flag{j}{2}}{\flag{n}{i+1}} \norm{v| E} h( \tilde{\MSO}^E_2 K, \pr_{\! E} v ),
		\end{equation*}
		where $\tilde{\MSO}_2^E$ denotes the centered mean section operator $\tilde{\MSO}_2$ relative to $E$.
		By~\cref{eq:MSO_Berg_fct},
		\begin{equation*}
			 h( \tilde{\MSO}^E_2 K, \pr_{\! E} v )
			 = m_{i+1,2} \int_{\S^i(E)} g_2(\langle \pr_{\! E} v,w \rangle) S_i^E(K,dw).
		\end{equation*}
		In conclusion, for every $(E,u)\in\Fl_{n,i+1}$ and $v\in\S^{n-1}$,
		\begin{equation*}
			\KS_{\tilde{\MSO}_j}((E,u),v)
			= \frac{\flag{j}{2}}{\flag{n}{i+1}} m_{i+1,2} \norm{v|E} g_2(\langle \pr_{\! E}v,u\rangle).
		\end{equation*}		
	\end{exl}
	
	Next, we show that the Klain--Schneider function of a Minkowski valuation can be expressed in terms of its generating function via some hemispherical transform. As an application, we can describe how the Klain--Schneider function is transformed under composition with Minkowski endomorphisms.
	
	\begin{thm}\label{KS_gen_fct}
		Let $1\leq i\leq n-1$ and $\Phi\in\MVal_i$. Then for $(E,u)\in\Fl_{n,i+1}$,
		\begin{equation}\label{eq:KS_gen_fct}
			\KS_\Phi((E,u),{}\cdot{})
			= \frac{1}{\binom{n-1}{i}} \lambda_{\H^{n-i-1}(E,u)} \ast f_{\Phi},
		\end{equation}
		where $\lambda_{\H^{n-i-1}(E,u)}$ denotes the restriction of  $\mathcal{H}^{n-i-1}$ to $\H^{n-i-1}(E,u)$.
	\end{thm}
	\begin{proof}
		Suppose first that (the associated real valued valuation of) $\Phi$ is smooth, and thus, $f_\Phi$ is smooth. Then \cref{eq:KS_spherical} yields for all $\vartheta\in\SO(n)$,
		\begin{align*}
			&\KS_\Phi((E,u),\vartheta e_n)
			= \frac{1}{\binom{n-1}{i}} (\pi_{E,-i}(\vartheta f_{\Phi}))(u)
			= \frac{1}{\binom{n-1}{i}} \int_{\H^{n-i-1}(E,u)} (\vartheta f_\Phi)(w) dw \\
			&\qquad = \frac{1}{\binom{n-1}{i}} \int_{\S^{n-1}} (\vartheta f_{\Phi})( w) ~\lambda_{\H^{n-i-1}(E,u)}(dw)
			= \frac{1}{\binom{n-1}{i}} \big(\lambda_{\H^{n-i-1}(E,u)} \ast f_{\Phi}\big)(\vartheta e_n).
		\end{align*}
		
		To pass from the smooth to the general case, we follow the argument of the proof of \cite[Theorem~6.5]{Schuster2010}. Let $\eta_\delta \in C^\infty(\S^{n-1})$ be a spherical approximate identity of zonal functions. Then $f_{\Phi}\ast\eta_\delta$ converges to $f_\Phi$ in $L^1(\S^{n-1})$ as $\delta \to 0$. Moreover, the functions $f_{\Phi}\ast \eta_\delta$ generate Minkowski valuations $\Phi^{(\delta)}\in \MVal_i$ and $\Phi^{(\delta)}$ converges to $\Phi$ uniformly on compact subsets of $\K(\R^n)$. Thus, by continuity of the Klain--Schneider map, $\KS_{\Phi^{(\delta)}}$ converges uniformly to $\KS_{\Phi}$. 	
	\end{proof}
	
	\begin{cor}\label{MEnd_KS}
		Let $1\leq i\leq n-1$ and $\Phi\in\MVal_i$. Then for every Minkowski endomorphism $\Psi$ and $(E,u)\in\Fl_{n,i+1}$,
		\begin{equation}\label{eq:MEnd_KS}
			\KS_{\Psi\circ\Phi}((E,u),{}\cdot{})
			= \KS_{\Phi}(E,u),{}\cdot{}) \ast \mu_{\Psi}.
		\end{equation}
	\end{cor}
	\begin{proof}
		By combining \cref{eq:MEnd_gen_fct} and \cref{eq:KS_gen_fct},
		\begin{align*}
			&\KS_{\Psi\circ\Phi}((E,u),{}\cdot{})
			= \frac{1}{\binom{n-1}{i}} \lambda_{\H^{n-i-1}(E,u)} \ast f_{\Psi\circ\Phi}
			= \frac{1}{\binom{n-1}{i}} \lambda_{\H^{n-i-1}(E,u)} \ast (f_{\Phi}\ast\mu_{\Psi}) \\
			&\qquad = \frac{1}{\binom{n-1}{i}} (\lambda_{\H^{n-i-1}(E,u)} \ast f_{\Phi}) \ast\mu_{\Psi}
			= \KS_{\Phi}((E,u),{}\cdot{}) \ast \mu_{\Psi},
		\end{align*}
		and thus, we obtain \cref{eq:MEnd_KS}.
	\end{proof}
	
	This corollary shows that composition with Minkowski endomorphisms acts on the Klain--Schneider function of a Minkowski valuation in the same way as on its generating function: by convolution with the corresponding generating measure from the right. Consequently, the discussion of the structure of the space of generating functions in \cref{sec:MVali} also translates to Klain--Schneider functions of Minkowski valuations.

	\appendix

	\section{The spherical convolution}
	\label{sec:convol}

	We can identify $\S^{n-1}$ with the homogeneous space $\SO(n)/\SO(n-1)$, where we denote by $\SO(n-1)\subseteq \SO(n)$ the the subgroup of rotations that stabilize the north pole~$e_n$. Due to this identification, the convolution structure on the compact Lie group~$\SO(n)$ naturally induces a convolution structure on $\S^{n-1}$. For an in-depth exposition, we recommend the article by Grinberg and Zhang~\cite{Grinberg1999} and the book by Takeuchi~\cite{Takeuchi1994}.
	
	The convolution of some $\phi\in C^{\infty}(\S^{n-1})$ and $\gamma\in C^{-\infty}(\S^{n-1})^{\mathrm{zonal}}$ is defined by
	\begin{equation*}
		(\phi\ast\gamma)(\vartheta e_n)
		= \langle \phi,  \vartheta\gamma\rangle_{\S^{n-1}}
		= \langle\vartheta^{-1}\phi ,\gamma \rangle_{\S^{n-1}},
		\qquad \vartheta\in\SO(n).
	\end{equation*}
	Since $\SO(n)$ operates transitively on the unit sphere and $\gamma$ is zonal, this is well-defined and it turns out that $\phi\ast\gamma\in C^{\infty}(\S^{n-1})$. Thus, we may define the convolution of spherical distributions by
	\begin{gather*}
		C^{-\infty}(\S^{n-1}) \times C^{-\infty}(\S^{n-1})^{\mathrm{zonal}} \to C^{-\infty}(\S^{n-1}):
		(\nu,\gamma) \mapsto \nu\ast\gamma \\
		\langle \phi,\nu\ast\gamma \rangle_{\S^{n-1}}
		= \langle \phi\ast\gamma,\nu \rangle_{\S^{n-1}},
		\qquad \phi\in C^{\infty}(\S^{n-1}).
	\end{gather*}
	Clearly, the convolution product is linear in each of its arguments. Although its definition is fundamentally asymmetrical (in that the right factor must always be zonal), it enjoys the nice properties one might expect. For all distributions $\nu\in C^{\infty}(\S^{n-1})$ and $\gamma,\gamma_1,\gamma_2\in C^{-\infty}(\S^{n-1})^{\mathrm{zonal}}$, we have the following:
	\begin{itemize}
		\item $(\nu\ast\gamma_1)\ast\gamma_2=\nu\ast(\gamma_1\ast\gamma_2)$.
		\hfill (associativity)
		\item $\gamma_1\ast\gamma_2=\gamma_2\ast\gamma_1$.
		\hfill (commutativity)
		\item $\nu\ast\delta_{e_n}=\nu$ and $\delta_{e_n}\ast\gamma=\gamma$.
		\hfill (neutral element)		
		\item $(\vartheta\nu)\ast\gamma=\vartheta(\nu\ast\gamma)$ for all $\vartheta\in\SO(n)$.
		\hfill ($\SO(n)$-equivariance)		
	\end{itemize}
	
	Next, note that by virtue of the chain of identifications \cref{eq:chain_of_identifications}, the definition above also encompasses the convolution of continuous functions, Lebesgue integrable functions, and signed measures.
	We want to note that the spherical convolution of two signed measures is again a signed measure. Also, the convolution of a signed measure and an $L^1(\S^{n-1})$ function is an $L^1(\S^{n-1})$ function. In particular, identities such as \cref{eq:gen_fct} are to be understood in a weak sense, as an equality in the space $L^1(\S^{n-1})$.
	
	Moreover, the convolution of a signed measure and a continuous function is again a continuous function that can be expressed by an integral representation. If $f\in C(\S^{n-1})$ and $\mu\in\mathcal{M}(\S^{n-1})^{\mathrm{zonal}}$, then
	\begin{equation*}
		(f\ast\mu)(\vartheta e_n)
		= \int_{\S^{n-1}} f(\vartheta v)~\mu(dv),
		\qquad \vartheta\in\SO(n).
	\end{equation*}
	In the case where $\mu\in\mathcal{M}(\S^{n-1})$ and $f\in C(\S^{n-1})^{\mathrm{zonal}}$, there is a unique function $\bar{f} \in C[-1,1]$ such that $f(w)=\bar{f}(\langle e_n,w\rangle)$ for all $w\in\S^{n-1}$, and then,
	\begin{equation*}
		(\mu\ast f)(u)
		= \int_{\S^{n-1}} \bar{f}(\langle u,v\rangle) ~\mu(dv),
		\qquad u\in\S^{n-1}.
	\end{equation*}

	\section{Spherical Harmonics}
	\label{sec:harmonics}
	
	As a general reference for this section, we cite the monograph by Groemer~\cite{Groemer1996}.
	Denote by $\mathcal{H}^n_k$ the space of \emph{spherical harmonics} of dimension $n\geq3$ and degree $k\geq 0$, that is, harmonic, $k$-homogeneous polynomials, restricted to the unit sphere. These turn out to be precisely the eigenspaces of the spherical Laplacian, that is, $\Delta_{\S^{n-1}}Y_k=-k(k+n-2)Y_k$ for every $Y_k\in\mathcal{H}^n_k$. Recall that $\Delta_{\S^{n-1}}$ is a second-order uniformly elliptic self-adjoint differential operator that has compact resolvent and intertwines rotations. Consequently, 
	$L^2(\S^{n-1})$ decomposes into an orthogonal direct sum of the spaces $\mathcal{H}^n_k$, each of which is finite dimensional and $\SO(n)$-irreducible. 
	
	For each $k\geq 0$, the subspace of zonal spherical harmonics in $\mathcal{H}^n_k$ is of dimension one and spanned by $P^n_k(\langle e_n,{}\cdot{}\rangle)$, where $P^n_k$ denotes the Legendre polynomial of dimension $n\geq 3$ and degree $k\geq 0$. It can be defined by \emph{Rodrigues' formula}, which states that
	\begin{equation}\label{eq:Rodrigues_formula}
		P^n_k(t)
		= (-1)^k \frac{\Gamma\big(\frac{n-1}{2}\big)}{2^k\Gamma\big(\frac{n-1}{2} + k\big)} (1-t^2)^{-\frac{n-3}{2}} \left(\frac{d}{dt} \right)^k (1-t^2)^{\frac{n-3}{2}+ k}.
	\end{equation}
	The orthogonal projection $\pi_k$ from $L^2(\S^{n-1})$ onto the subspace $\mathcal{H}^n_k$ is given by
	\begin{equation*}
		\pi_k \phi
		= \tfrac{\dim\mathcal{H}^n_k}{\omega_{n-1}}~ \phi\ast P^n_k(\langle e_n,{}\cdot{}\rangle),
		\qquad \phi\in L^2(\S^{n-1}).
	\end{equation*}
	As a convolution transform, $\pi_k$ extends to a map from $C^{-\infty}(\S^{n-1})$ onto $\mathcal{H}^n_k$, and for every distribution $\nu\in C^{-\infty}(\S^{n-1})$, its formal Fourier series $\sum_{k=0}^\infty \pi_k\nu$, also called its \emph{spherical harmonic expansion}, converges to $\nu$ in the weak sense. Moreover, if $\gamma\in C^{-\infty}(\S^{n-1})^{\mathrm{zonal}}$, then
	\begin{equation*}
		\gamma
		= \sum_{k=0}^\infty \tfrac{\dim\mathcal{H}^n_k}{\omega_{n-1}} a^n_k[\gamma] P^n_k(\langle e_n,{}\cdot{}\rangle),
	\end{equation*}
	where $a^n_k[\gamma]=\langle P^n_k(\langle e_n,{}\cdot{}\rangle), \gamma \rangle_{\S^{n-1}}$.
	The \emph{Funk--Hecke Theorem} states that the convolution product of two distributions $\nu\in C^{-\infty}(\S^{n-1})$ and $\gamma\in C^{-\infty}(\S^{n-1})^{\mathrm{zonal}}$ has the spherical harmonic expansion
	\begin{equation*}
		\nu\ast\gamma
		= \sum_{k=0}^\infty a^n_k[\gamma] \pi_k \nu.
	\end{equation*}
	That is, the convolution transform $\nu\mapsto \nu\ast\gamma$ acts as a multiple of the identity on each space $\mathcal{H}^n_k$ of spherical harmonics. The Fourier coefficients $a^n_k[\gamma]$ are thus also called \emph{multipliers}. 
	
	If $\gamma$ is defined by some $L^1(\S^{n-1})$~function, then the multipliers can be computed using cylindrical coordinates on the sphere: If $g:(-1,1)\to\R$ is measurable and $(1-t^2)^{\frac{n-3}{2}}g(t)$ is integrable on $(-1,1)$, then $g(\langle e_n,{}\cdot{}\rangle)\in L^1(\S^{n-1})^{\mathrm{zonal}}$ and 
	\begin{equation}\label{eq:multipliers_interval}
		a^n_k[g(\langle e_n,{}\cdot{}\rangle)]
		= \omega_{n-1} \int_{(-1,1)} P^n_k(t) (1-t^2)^{\frac{n-3}{2}} g(t) dt.
	\end{equation}
	For convenience, we use the convention $a^n_k[g]=a^n_k[g(\langle e_n,{}\cdot{}\rangle)]$.
	
	Finally, we prove \cref{multiplier_diff}, which is a refinement of \cite[Lemma~3.6]{OrtegaMoreno2023}. To that end, we need the following technical integration by parts lemma.
	
	\begin{lem}[{\cite[Lemma~3.3]{Brauner2023}}]\label{alpha_int_by_parts}
		Let $\alpha>0$ and let $g\in\D'(-1,1)$ be such that $(1-t^2)^{\frac{\alpha}{2}}g'(t)$ is a finite signed measure on $(-1,1)$. Then $g$ is a locally integrable function and $(1-t^2)^{\frac{\alpha-2}{2}}g(t)\in L^1(-1,1)$. Moreover, whenever $\psi\in C^1(-1,1)$ such that both $(1-t^2)^{-\frac{\alpha-2}{2}}\psi'(t)$ and $(1-t^2)^{\frac{\alpha}{2}}\psi(t)$ are bounded, then
		\begin{equation}\label{eq:alpha_int_by_parts}
			\int_{(-1,1)} \psi(t) g'(dt)
			= - \int_{(-1,1)} \psi'(t)g(t) dt.
		\end{equation}
	\end{lem}	
	
	\begin{proof}[Proof of \cref{multiplier_diff}]
		By \cref{alpha_int_by_parts}, we have that $(1-t^2)^{\frac{n-3}{2}}g(t)\in L^1(-1,1)$.
		Combining \cref{eq:multipliers_interval} with identity~\cref{eq:alpha_int_by_parts} in the instance where $\psi=P^{n+2}_k$ yields
		\begin{align*}
			&\frac{1}{\omega_{n+1}}a^{n+2}_k[g']
			=\int_{(-1,1)} P^{n+2}_k(t) (1-t^2)^{\frac{n-1}{2}} g'(t) dt \\
			&\qquad = (n-1) \int_{(-1,1)} P^n_{k+1}(t) (1-t^2)^{\frac{n-3}{2}} g(t) dt
			= \frac{n-1}{\omega_{n-1}} a^n_{k+1}[g],
		\end{align*}
		where in the second equality, we also employed Rodrigues' formula \cref{eq:Rodrigues_formula}.		
		Since $(n-1)\omega_{n+1}=2\pi\omega_{n-1}$, this yields \cref{eq:multiplier_diff}.
	\end{proof}

	\section{Proof of \texorpdfstring{\cref{Radon_linear}}{Proposition~3.2}}
	\label{sec:proofs}
	
	For the proof of \cref{Radon_down_linear}, we express the Radon type transform $\RT_{k,k-1}$ in terms of an actual Radon transform of some weighted spherical projections. To that end, if $E\subseteq F\subseteq \R^n$ are nested subspaces, we denote by $\pi_{E,m}^F$ and $\pi_{E,m}^{F \ast}$ the $m$-weighted spherical projection and lifting relative to $F$. Moreover, if $F\in \Gr_{n,k}$ and $E\in\Gr_{n,k-1}^F$, then for all $f\in C(\S^{k-1}(F))$
	\begin{equation}\label{eq:sphere_hyperplane_integration}
		\int_{\S^{k-1}(F)} f(u) du
		= \int_{\S^{k-2}(E)} \int_{\H^1(F;E,u)} f(v) \langle u,v\rangle^{k-2}dv ~du,
	\end{equation}
	where $\H^1(F;E,u)=\{ v\in\S^{k-1}(F)\setminus E^\perp: \pr_{\! E}v=u \}$ (cf.~\cite[(3.3)]{Goodey2011}).
	
	\begin{lem}\label{Radon_down_sph_proj}
		Let $1<k\leq n$. Then for all $\zeta\in C(\Fl_{n,k})$ and $(E,u)\in\Fl_{n,k-1}$,
		\begin{equation}\label{eq:Radon_down_sph_proj}
			[\RT_{k,k-1}\zeta](E,u)
			= \frac{\omega_{n-k+1}}{\omega_{n-k+2}} \int_{\Gr^E_{n,k}} [\pi^F_{E,n-2k+2}\zeta(F,{}\cdot{})](u) ~dF
		\end{equation}
	\end{lem}
	\begin{proof}
		First, we parametrize the Grassmann manifold $\Gr^{E\cap u^\perp}_{n,k-1}$ through setting $F=(E\cap u^\perp)\vee w$ for $F\in \Gr^{E\cap u^\perp}_{n,k-1}$ and $w\in \S^{n-k+1}(E^\perp\vee u)$. This yields
		\begin{align*}
			& [\RT_{k,k-1}\zeta](E,u)
			= \frac{1}{\omega_{n-k+2}} \int_{\S^{n-k+1}(E^\perp\vee u)} \zeta(E\vee w,\pr_{\! w^\perp} u) dw \\
			&\qquad = \frac{1}{\omega_{n-k+2}} \int_{\S^{n-k}(E^\perp)} \int_{\H^1(E^\perp\vee u;E^\perp,w)} \zeta(E\vee v,\pr_{\! v^\perp} u) \langle w,v\rangle ^{n-k}dv ~dw \\
			&\qquad = \frac{1}{\omega_{n-k+2}} \int_{\S^{n-k}(E^\perp)} \int_{\S^1(u\vee w)\cap w^+} \zeta(E\vee w,\pr_{\! v^\perp} u) \langle u , \pr_{\! v^\perp} u\rangle ^{n-k}dv ~dw,
		\end{align*}
		where the second equality is an application of \cref{eq:sphere_hyperplane_integration} relative to the space $E^\perp\vee u$ and in the final equality, we simply rewrote the inner integral. Next, observe that for every vector $w\in\S^{n-k}(E^\perp)$ and function $f \in C(\S^1(u\vee w))$,
		\begin{align*}
			\int_{\S^1(u\vee w)\cap w^+} f(\pr_{\! v^\perp} u) dv
			= \int_{\S^1(u\vee w)\cap u^+} f(v) dv
			= \int_{\H^1(E\vee w;E,u)} f(v) dv,
		\end{align*}
		as follows from applying a change of variables with respect to a rotation by $\pi/2$ in the plane $u\vee w$. Applying this to $f(v)=\zeta(E\vee w,v) \langle u,v\rangle ^{n-k}$ yields		
		\begin{align*}
			& [\RT_{k,k-1}\zeta](E,u)
			= \frac{1}{\omega_{n-k+2}} \int_{\S^{n-k}(E^\perp)} \int_{\H^1(E\vee w;E,u)} \zeta(E\vee w,v) \langle u,v\rangle^{n-k}dv ~dw \\
			&\qquad = \frac{\omega_{n-k+1}}{\omega_{n-k+2}} \int_{\Gr^E_{n,k}} \int_{\H^1(F;E,u)} \zeta(F,v) \langle u,v\rangle^{n-k} dy ~dF,
		\end{align*}
		where in the second equality, we applied the parametrization $F=E\vee w$ with $F\in \Gr^E_{n,k}$ and $w\in \S^{n-k}(E^\perp)$.
	\end{proof}

	We will need the fact that weighted spherical projections map linear functions to linear functions, which was shown in~\cite{Goodey2014}.
	
	\begin{lem}[{\cite[Lemma~2.3]{Goodey2014}}] \label{sph_proj_linear}
		Let $1\leq k\leq n$, $E\in\Gr_{n,k}$, and $m>-k$. Then for all $u\in\S^{n-1}$ and $v\in\S^{k-1}(E)$,
		\begin{equation}\label{eq:sph_proj_linear}
			[\pi_{E,m}\langle u,{}\cdot{} \rangle](v)
			= \frac{\omega_{n+m+1}}{\omega_{k+m+1}} \langle u,v\rangle.
		\end{equation}
	\end{lem}

	\begin{proof}[Proof of \cref{Radon_linear}~\ref{Radon_down_linear}]
		Let $\zeta\in C(\Fl_{n,k})$ and denote by $z:\Gr_{n,k}\to\R^n$ the continuous function with the property that $\pi^F_1\zeta(F,{}\cdot{})=\langle z(F),{}\cdot{}\rangle$ for $F\in \Gr_{n,k}$.
		By combining \cref{eq:Radon_down_sph_proj} and \cref{eq:sph_proj_linear}, we have that for all $(E,u)\in \Fl_{n,k-1}$,		
		\begin{align*}
			&\frac{\omega_{n-k+2}}{\omega_{n-k+1}} [\RT_{k,k-1}\pi_1^{\langle k\rangle }\zeta](E,u)
			= \int_{\Gr_{n,k}^{E}} \!\! [\pi^{F}_{E,n-2k+2}\langle z(F),{}\cdot{}\rangle](u) dF \\
			&\qquad = \frac{\omega_{n-k+3}}{\omega_{n-k+2}} \int_{\Gr_{n,k}^{E}} \!\! \langle z(F),u\rangle dF
			= \frac{\omega_{n-k+3}}{\omega_{n-k+2}}\frac{1}{\omega_k} \int_{\Gr_{n,k}^{E}} \int_{\S^{k-1}(F)} \!\! \zeta(F,v)\langle u,v\rangle dv ~dF,
		\end{align*}
		Note that whenever $F\in\Gr_{n,k}^E$, then $[\pi^{F\ast}_{E,1}\langle u,{}\cdot{}\rangle](v) = \langle u,v \rangle$ for all $u\in\S^{k-2}(E)$ and $v\in\S^{k-1}(F)$, and thus,
		\begin{align*}
			&\int_{\S^{k-1}(F)} \! \zeta(F,v)\langle u,v\rangle dv
			= \int_{\S^{k-1}(F)} \! \zeta(F,v)[\pi^{F\ast}_{E,1}\langle u,{}\cdot{}\rangle](v) dv \\
			&\qquad = \int_{\S^{k-2}(E)} [\pi^{F}_{E,1}\zeta(F,{}\cdot{}) ](v) \langle u,v\rangle dv.
		\end{align*}
		By changing the order of integration, we obtain that
		\begin{align*}
			&[\RT_{k,k-1}\pi_1^{\langle k\rangle }\zeta](E,u)
			= \frac{\omega_{n-k+1}\omega_{n-k+3}}{\omega_{n-k+2}^2} \frac{1}{\omega_k} \int_{\S^{k-2}(E)} \int_{\Gr_{n,k}^{E}} \!\! [\pi^{F}_{E,1}\zeta(F,{}\cdot{}) ](v) dF~  \langle u,v\rangle dv \\
			&\qquad= \frac{1}{\omega_{k-1}} \int_{\S^{k-2}(E)} [\RT_{k,k-1}'\zeta](E,v) \langle u,v\rangle dv
			= [\pi_1^{\langle k-1\rangle }\RT_{k,k-1}'\zeta](E,u),
		\end{align*}
		which concludes the argument.
	\end{proof}
	
	Similar as for $\RT_{k,k-1}$, we can express the Radon type transform $\RT_{k,k+1}$ in terms of an actual Radon transform and the $1$-weighted spherical lifting. More precisely, if $1\leq k<n$, then for all $\zeta\in C(\Fl_{n,k})$ and $(E,u)\in \Fl_{n,k+1}$,
	\begin{equation}\label{eq:Radon_up_sph_lift}
		[\RT_{k,k+1}\zeta](E,u)
		= \int_{\Gr_{n,k}^{E}} [\pi_{F,1}^{E \ast}\zeta(F,{}\cdot{})](u)~dF,
	\end{equation}
	which is immediate from the definition. Now we prove \cref{Radon_up_linear} of \cref{Radon_linear}.
	
	\begin{proof}[Proof of \cref{Radon_linear}~\ref{Radon_up_linear}]
		Let $\zeta\in C(\Fl_{n,k})$ and $(E,u)\in\Fl_{n,k+1}$. Then by \cref{eq:Radon_up_sph_lift} and a change of order of integration,
		\begin{align*}
			&[\pi_1^{\langle k+1 \rangle}\RT_{k,k+1}\zeta](E,u)
			= \frac{1}{\omega_{k+1}} \int_{\Gr_{n,k}^{E}} \int_{\S^{k}(E)} [\pi_{F,1}^{E \ast}\zeta(F,{}\cdot{})](v) \langle u,v\rangle dv ~dF \\
			&\qquad = \frac{1}{\omega_{k+1}} \int_{\Gr_{n,k}^{E}} \int_{\S^{k-1}(F)} \zeta(F,v) [\pi_{F,1}^{E}\langle u,{}\cdot{}\rangle](v)dv ~dF \\
			&\qquad = \frac{\omega_{k+3}}{\omega_{k+2}}\frac{1}{\omega_{k+1}} \int_{\Gr_{n,k}^{E}} \int_{\S^{k-1}(F)} \zeta(F,v) \langle u,v\rangle dv ~dF,
		\end{align*}
		where the final equality is due to \cref{eq:sph_proj_linear}. Note that for every $F\in\Gr_{n,k}^{E}$,
		\begin{align*}
			\frac{1}{\omega_{k}} \int_{\S^{k-1}(F)} \zeta(F,v) \langle u,v\rangle dv
			= [\pi_1^{\langle k\rangle}\zeta](F,u)
			= [\pi_{F,1}^{E \ast} (\pi_1^{\langle k\rangle } \zeta)(F,{}\cdot{})](u).
		\end{align*}
		Plugging this into the expression above and applying \cref{eq:Radon_up_sph_lift} again, we obtain
		\begin{align*}
			&[\pi_1^{\langle k+1\rangle }\RT_{k,k+1}\zeta](E,u)
			= \frac{\omega_{k+3}}{\omega_{k+2}}\frac{\omega_{k}}{\omega_{k+1}} \int_{\Gr_{n,k}^{E}} [\pi_{F,1}^{E \ast} (\pi_1^{\langle k\rangle } \zeta)(F,{}\cdot{})](u) ~dF \\
			&\qquad = \frac{k}{k+1} [\RT_{k,k+1}\pi_1^{\langle k\rangle } \zeta](E,u),
		\end{align*}
		which concludes the argument.
	\end{proof}



	\section*{Acknowledgments}
	
	The first- and third-named author were supported by the Austrian Science Fund (FWF), project number: P31448-N35.
	The third-named author was supported by the Austrian Science Fund (FWF), project number: ESP~236 ESPRIT program.

	\begingroup
	\bibliographystyle{abbrv}
	\bibliography{references}{}
	\endgroup

\end{document}